\documentclass[11pt,reqno]{amsart} 
\usepackage[T1]{fontenc}
\usepackage{amsmath,amssymb,amsthm,amsfonts,stmaryrd,rotating,dsfont,color}
\usepackage[latin9]{inputenc} 
\usepackage[a4paper]{geometry}
\usepackage[matrix,arrow,curve]{xy}
\usepackage{tensor}
\usepackage{color}

\newcommand{\new}{}
\newcommand{\old}{}

\usepackage{hyperref}
\usepackage[draft,margin]{fixme}

\newcommand{\GpqG}{G {_{p}\times_{q}} G}

\newcommand{\GstG}{G {_{s}\times_{t}} G}
\newcommand{\GssG}{G {_{s}\times_{s}} G}
\newcommand{\GttG}{G {_{t}\times_{t}} G}

\newcommand{\sindex}[3]{}

 \usepackage[pagewise,displaymath,mathlines]{lineno}

\numberwithin{equation}{section}

\theoremstyle{definition} 

\swapnumbers
\newtheorem{remark}{Remark}[section]
\newtheorem{remarks}[remark]{Remarks}

\newtheorem{notation}[remark]{Notation}
\newtheorem*{notation*}{Notation}

\theoremstyle{plain}
\newtheorem{definition}[remark]{Definition}
\newtheorem{theorem}[remark]{Theorem}
\newtheorem{proposition}[remark]{Proposition}
\newtheorem{corollary}[remark]{Corollary}
\newtheorem{lemma}[remark]{Lemma}

\newtheorem*{assumption*}{Assumption}

\newcommand{\C}{\mathds{C}}

\newcommand{\gr}[5]{\tensor*[^{#2}_{#3}]{#1}{^{#4}_{#5}}}

\newcommand{\Deltalt}{\Delta_{C}}
\newcommand{\Deltart}{\Delta_{B}}

\newcommand{\id}{\iota}
\DeclareMathOperator{\Hom}{Hom}
\DeclareMathOperator{\End}{End}
\DeclareMathOperator{\lspan}{span}

\newcommand{\bB}{{_{B}B}}
\newcommand{\Bb}{B_{B}}
\newcommand{\cC}{{_{C}C}}
\newcommand{\Cc}{C_{C}}

\newcommand{\bsA}{{_{B}A}}
\newcommand{\btA}{A^{B}}
\newcommand{\bAs}{A_{B}}
\newcommand{\bAt}{{^{B}A}}

\newcommand{\csA}{{_{C}A}}
\newcommand{\ctA}{A^{C}}
\newcommand{\cAs}{A_{C}}
\newcommand{\cAt}{{^{C}A}}

\newcommand{\AltkA}{\ctA \overline{\times} \csA} 
\newcommand{\ArtkA}{\bAs \overline{\times} \bAt} 

\newcommand{\oo}{\otimes}

\newcommand{\AlA}{\ctA \otimes \csA}

\newcommand{\ArA}{\bAs \otimes \bAt}

\newcommand{\ATAsA}{\btA \otimes {^{B}A_{B}} \otimes
{_{B}A}}
\newcommand{\ATAlA}{\btA \otimes {^{B}A^{C}} \otimes \csA}
\newcommand{\AlAsA}{\ctA \otimes \csA_{B} \otimes \bsA}
\newcommand{\AlAlA}{\ctA \otimes  \csA^{C} \otimes \csA}

\newcommand{\ArArA}{\bAs \oo {^{B}A_{B}} \oo \bAt}

\newcommand{\bA}{{_{B}A}}
\newcommand{\Ab}{A_{B}}
\newcommand{\cA}{{_{C}A}}
\newcommand{\Ac}{A_{C}}
\newcommand{\sbA}{A^{B}}
\newcommand{\Asb}{\gr{A}{B}{}{}{}}
\newcommand{\scA}{A^{C}}
\newcommand{\Asc}{\gr{A}{C}{}{}{}}

\newcommand{\AbA}{\Ab \oo \bA}

\newcommand{\AcA}{\Ac \oo \cA}

\newcommand{\AcAbA}{\Ac \oo \cA_{B} \oo \bA}

\newcommand{\ArAbA}{\Ab \oo \Asb_{B} \oo \bA}
\newcommand{\ArAlA}{\bAs \oo {^{B}A^{C}} \oo \csA}

\newcommand{\AlAbA}{\scA \oo \cA_{B} \oo \bA}

\newcommand{\AlArA}{\ctA \otimes {_{C}A_{B}} \oo \bAt}

\newcommand{\op}{\mathrm{op}}
\newcommand{\co}{\mathrm{co}}

\newcommand{\Tl}{T_{\lambda}}
\newcommand{\Tr}{T_{\rho}}
\newcommand{\lT}{{_{\lambda}T}}
\newcommand{\rT}{{_{\rho}T}}

\newcommand{\Sl}{S_{\lambda}}
\newcommand{\Sr}{S_{\rho}}
\newcommand{\lS}{{_{\lambda}S}}

\newcommand{\epsh}{\varepsilon_{H}}
\newcommand{\epslt}{{_{C}\varepsilon}}
\newcommand{\epsrt}{\varepsilon_{B}}

\newcommand{\Rl}{\Tl^{-1}}
\newcommand{\lR}{\lT^{-1}}
\newcommand{\Rr}{\Tr^{-1}}

\newcommand{\mult}{m}
\newcommand{\es}{\varepsilon_{s}} 
\newcommand{\et}{\varepsilon_{t}} 
\newcommand{\Es}{E_{s}} 
\newcommand{\Et}{E_{t}} 

\newcommand{\actleft}{\triangleleft}

\newcommand{\actright}{\triangleright}

\title[Multiplier Hopf
  algebroids.]{Multiplier Hopf
  algebroids. Basic theory and examples.}

\author{Thomas Timmermann} \address{FB Mathematik und Informatik, University of Muenster
  \\ Einsteinstr.\ 62, 48149 Muenster, Germany} \email{timmermt@math.uni-muenster.de}

\author{Alfons Van Daele} \address{Department of Mathematics, University of Leuven,
  Celestijnenlaan 200B, B-3001 Heverlee, Belgium} \email{Alfons.VanDaele@wis.kuleuven.be}
\thanks{Supported by the SFB 878 ``Groups, geometry and actions'' funded by the DFG}
\date{\today}


\begin{document}


\begin{abstract} Multiplier Hopf algebroids are algebraic versions of quantum groupoids that generalize Hopf algebroids to the non-unital case and weak (multiplier) Hopf algebras to non-separable base algebras.  The main structure maps of a multiplier Hopf algebroid are a left and a right comultiplication. \new We show that bijectivity of two associated canonical maps is equivalent to the existence of an  antipode,  discuss invertibility of the antipode, and  present some examples and special cases. \old
\end{abstract}

\maketitle

\tableofcontents

\section{Introduction}

\label{section:introduction}

 Quantum groupoids have appeared  in a variety of guises and mathematical contexts, for example, as generalized Galois symmetries for depth 2 inclusions of factors or algebras \cite{boehm:extensions}, \cite{enock:inclusions3}, \cite{enock:actions}, \cite{kadison:extensions}, \cite{kadison:inclusions}, \cite{nikshych:inclusions}, as dynamical quantum groups in connection with solutions to the quantum dynamical Yang-Baxter equation \cite{donin}, \cite{etingof:qdybe}, \cite{koelink:su2}, or as Tannaka-Krein duals of certain tensor categories of bimodules \cite{hai:tannaka}, \cite{mccurdy:tannaka}, \cite{pfeiffer:tannaka}. Common to all approaches are the basic constituents of a quantum groupoid --- a pair of anti-isomorphic algebras $B$ and $C$ with homomorphisms into an algebra $A$ together with a comultiplication on $A$ that takes values in a certain fiber product $A \ast A$ involving $B$ and $C$. These ingredients are, in a sense, dual to the constituents of a groupoid, and satisfy corresponding conditions like co-associativity of the comultiplication.

In this article, we extend the existing algebraic approaches to quantum groupoids via Hopf algebroids \cite{boehm:hopf}, \cite{boehm:bijective}, \cite{lu:hopf,xu}, weak Hopf algebras \cite{boehm:weak1}, \cite{vainer}, \cite{schauenburg:comparison} and weak multiplier Hopf algebras \cite{boehm:weak}, \cite{daele:weakmult0}, \cite{daele:weakmult} by considering so-called \new multiplier Hopf algebroids\old, where the underlying algebras are no longer assumed to be unital.

The motivation to study multiplier versions of Hopf algebroids and weak Hopf algebras is two-fold. First, there are natural examples which exhibit all features of a quantum groupoid except that the underlying algebras are not unital and can not be made unital in a natural way, like algebras of functions on non-compact groupoids.  Second, such examples appear as generalized Pontryagin duals of unital Hopf algebroids or weak Hopf algebras, and as in the case of Hopf algebras, one has to pass to a multiplier version to obtain a good duality theory beyond finite-dimensional cases \cite{daele:1}. In \cite{MR3607289}, we show that the multiplier Hopf algebroids introduced in this article provide a good algebraic setting for a generalised Pontryagin duality theory for quantum groupoids.

The theory of (multiplier) Hopf algebroids and the theory of weak (multiplier) Hopf algebras differ mainly in the target of the comultiplication and both have their advantages and draw-backs.  Weak (multiplier) Hopf algebras may be easier to work with, but their base algebras are automatically separable and, in particular, semi-simple; see Proposition 2.11 \cite{boehm:weak1}.  Multiplier Hopf algebroids overcome this restriction and are not only more general, but also, in a sense, more natural than weak multiplier Hopf algebras.  They may, however, appear more difficult because they involve two versions of the comultiplication simultaneously, as will be explained below. In the finite-dimensional case, both approaches are equivalent \cite{nikshych:algversions,schauenburg:comparison}. In \cite{daele:relation} and \cite{daele:modified}, we show that every regular weak multiplier Hopf algebra gives rise to a regular multiplier Hopf algebroid, but that even in the case where the base algebras are separable, the converse is not true.

\new

\medskip

Let us explain the main result of this article in some more detail. \old

Similarly like a bialgebroid, a \emph{multiplier bialgebroid} is given by a total algebra $A$, two base algebras $B,C$ with anti-isomorphisms $B\rightleftarrows C$, and a left and a right comultiplication $\Deltalt$ and $\Deltart$, respectively, related by a mixed co-associativity condition.  In the unital case, these comultiplications take values in the left and the right Takeuchi product, respectively.  In the non-unital case, the latter have to be replaced by certain left or right multiplier algebras such that all products of the form
   \begin{align*} \Deltalt(b)(a \otimes 1), && \Deltalt(a)(1 \otimes
b), && (a \otimes 1)\Deltart(b), && (1 \otimes b)\Deltart(a),
  \end{align*}
  where $a,b\in A$, make sense as elements of certain tensor products of $A$ with itself relative to $B$ or $C$, respectively.  \new The definition of a left and a right counit then   carries over from the unital case.\old

  The main result of this article is \new  that a multiplier bialgebroid with a left and a right counit has an antipode  if and only if the canonical maps
    \begin{align} \label{eq:intro-lttr} 
  \lT \colon & a \otimes b
\mapsto (a \otimes 1)\Deltart(b), & \Tr
\colon & a \otimes b \mapsto \Deltalt(a)(1 \otimes b),
    \end{align}
    are bijective, where the ranges and domains are various tensor products of $A$ with itself relative to $B$ or $C$, respectively.  In that case, we call the multiplier bialgebroid a \emph{multiplier Hopf algebroid}. Its antipode is invertible if and only if the canonical maps
    \begin{align} \label{eq:intro-rttl}  \rT\colon & a \otimes b \mapsto (1 \otimes
b)\Deltart(a), & \Tl\colon & a \otimes b \mapsto \Deltalt(b)(a \otimes 1),
    \end{align}
are bijective as well, and in that case, we call the multiplier Hopf algebroid \emph{regular}.
\old
This result generalizes corresponding characterizations of multiplier Hopf algebras and Hopf
algebroids among multiplier bialgebras or bialgebroids; see \cite{daele:0} and Proposition 4.2 in
\cite{boehm:bijective}.  \new In the case of multiplier Hopf algebras,   bijectivity of the maps $\lT$ and $\Tr$ implies existence of a counit. In the case of multiplier bialgebroids, we can only prove existence of counits if  the maps $\rT$ and $\Tl$ are bijective as well. \old

The proof of the main result   uses only the canonical maps \new in \eqref{eq:intro-lttr}  and \eqref{eq:intro-rttl}  \old and a few key relations between them that are equivalent to multiplicativity, co-associativity and compatibility of the comultiplications $\Deltalt$ and $\Deltart$. To a large extent, we adopt and refine the arguments in \cite{daele:0}, but replace calculations involving the comultiplications by transparent commutative diagrams. This change of technique proves to be very helpful for keeping track of the module structures used for these tensor products and for ensuring that all maps involved are well-defined.  More importantly, this method makes explicit the key relations of the maps $\Tl,\Tr$ and $\lT,\rT$ used in the arguments and suggests to shift the perspective and to regard these canonical maps as the fundamental structure maps of a multiplier bialgebroid.

\medskip

This article is organized as follows.

In \S\ref{section:left}, we introduce multiplier analogues of left bialgebroids, which are given by algebras $A$ and $C$ with a homomorphism $s\colon C\to M(A)$, an anti-homomorphism $t\colon C\to M(A)$, and a left-sided comultiplication $\Deltalt$ from $A$ into a multiplier version of the Takeuchi product. The map $\Deltalt$ and its defining properties are described in terms of the canonical maps \new $\Tl$ and $\Tr$, see \eqref{eq:intro-lttr} and \eqref{eq:intro-rttl}, \old and various commutative diagrams, which will be used extensively later on. The notation used for these diagrams is explained in \ref{notation:left}.

In \S\ref{section:left-counits}, we introduce counits of left multiplier bialgebroids, and prove
uniqueness and existence in the case where the canonical maps are surjective or bijective, respectively.  In contrast to the unital case, we do not include existence of a
counit in the definition of a left multiplier bialgebroid, but consider them as additional
structure.

In \S\ref{section:right}, we turn to right multiplier bialgebroids and briefly summarize
the right-handed analogues of the left-handed concepts and results of
\S\ref{section:left} and \S\ref{section:left-counits}.

In \S\ref{section:hopf}, we come to the main result of this article, which is the definition and characterization of multiplier Hopf algebroids. We first formulate the necessary compatibility relation for a left and a right multiplier bialgebroid to form a two-sided multiplier bialgebroid and then \new show that existence of an antipode is equivalent to bijectivity of the canonical maps \eqref{eq:intro-lttr}. \old Along the way, we obtain many useful relations for the canonical maps and describe their inverses in terms of the antipode.

\new
In \S\ref{section:regular}, we show that the antipode is invertible if and only if the maps \eqref{eq:intro-rttl} are invertible as well, and derive further relations between the antipode and the canonical maps which hold in this case.

In \S\ref{section:examples}, we present several special cases and examples, including multiplier Hopf algebroids arising from weak multiplier Hopf algebras, multiplier Hopf $*$-algebroids, the function algebras and convolution algebras of \'etale groupoids, two-sided crossed products which generalize constructions in \cite{boehm:hopf}, \cite{vainer} and \cite{daele:weakmult2}, and proper, co-commutative and \'etale multiplier Hopf algebroids.
\old 

\medskip

We use the following conventions and terminology.

The identity map on a set $X$ will be denoted by $\iota_{X}$ or simply $\id$.  All algebras and modules will be complex vector spaces and all morphisms will be linear maps, but much of the theory developed in this article should apply in wider generality.

We denote the linear span of a subset $X$ of  a vector space $V$  by $\lspan X$.

Let $B$ be an algebra, not necessarily unital. We denote by $B^{\op}$\sindex{Aaop}{$A^{\op},B^{\op}$}{opposite algebra} the \emph{opposite algebra}, which has the same underlying vector space as $B$ but the reversed multiplication. When necessary, we write $b^{\op}$ when we regard an element $b\in B$ as an element of $B^{\op}$ to avoid confusion.

Given a right module $M$ over $B$, we write $M_{B}$ if we want to emphasize that $M$ is regarded as a right $B$-module. We call $M_{B}$ \emph{faithful} if for each non-zero $b\in B$ there exists an $m\in M$ such that $mb$ is non-zero, \emph{non-degenerate} if for each non-zero $m\in M$ there exists a $b \in B$ such that $mb$ is non-zero, \emph{idempotent} if $MB=M$, and we say that $M_{B}$ \emph{has local units in $B$} if for every finite subset $F\subset M$ there exists a $b\in B$ with $mb=m$ for all $m\in F$. Note that the last property implies the preceding two.

For left modules, we obtain the corresponding notation and terminology by identifying left $B$-modules with right $B^{\op}$-modules.

We write $B_{B}$ or ${_{B}B}$ when we regard $B$ as a right or left module over itself with respect to right or left multiplication. We say that the algebra $B$ is \emph{non-degenerate}, \emph{idempotent}, or \emph{has local units} if the modules ${_{B}B}$ and $B_{B}$ both are non-degenerate, idempotent or both have local units in $B$, respectively. Note that the last property again implies the preceding two.

\new
Working with non-unital algebras, we frequently need to use multipliers.

A \emph{left multiplier} of the algebra $B$ is a linear map $T \colon B \to B$ satisfying $T(bb')=T(b)b'$ for all $b,b'\in B$, that is, an endomorphism of the right $B$-module $B_{B}$.  We denote by $L(B):=\End(B_{B})$\sindex{L(B)}{$L(B)$}{left multiplier algebra} the algebra of all left multipliers of $B$.

A \emph{right multiplier} of the algebra $B$ is an endomorphism of the left $B$-module $_{B}B$.  When we think of such an endomorphism $T$ as a right multiplier, we write $bT$ instead of $T(b)$ for the image of $b\in B$ under $T$. We denote by $R(B):=\End({_{B}B})^{\op}$\sindex{R(B)}{$R(B)$}{right multiplier algebra} the algebra of right multipliers of $B$, so that $b(TS)=(bT)S$ for all $b\in B$ and $T,S\in R(B)$.

Note that $B_{B}$ or ${_{B}B}$ is non-degenerate if and only if the natural map from $B$ to $L(B)$ or $R(B)$, respectively, is injective.

Suppose that $B$ is non-degenerate. Then we define a \emph{multiplier} of $B$ to be a pair $T=(T_{l},T_{r})$ of maps $T_{l},T_{r} \colon B\to B$ satisfying $bT_{l}(b') = T_{r}(b)b'$ for all $b,b' \in B$. We write $Tb:=T_{l}(b)$ and $bT:=T_{r}(b)$ for all $b\in B$, so that the preceding equation takes the form $b(Tb')=(bT)b'$ for all $b,b' \in B$.  All multipliers form an algebra $M(B)$ with respect to the obvious addition and the multiplication given by $(T_{l},T_{r}) \circ (S_{l},S_{r}) = (T_{l} \circ S_{l},S_{r} \circ T_{r})$, that is, $(TS)b = T(Sb)$ and $b(TS)=(bT)S$ for all $b\in B$. A multiplier $T=(T_{l},T_{r})$ is uniquely determined by the components $T_{l}$ and $T_{r}$, which are a left and a right multiplier of $B$, respectively, so that $M(B)$ can be identified with subalgebras of $L(B)$ and $R(B)$.

More generally, if $B_{B}$ is non-degenerate, we define the multiplier algebra of $B$ to be the subalgebra $M(B) :=\{ T\in L(B) : BT\subseteq B\} \subseteq L(B)$,\sindex{M(B)}{M(B)}{multiplier algebra} where we identify $B$ with its image in $L(B)$. Likewise we define $M(B) =\{ T\in R(B) : BT \subseteq B\}$\sindex{M(B)}{M(B)}{multiplier algebra} if ${_{B}B}$ is non-degenerate, and both definitions coincide with the preceding one if $B_{B}$ and $_{B}B$ are non-degenerate.

\section{Left multiplier bialgebroids}
\label{section:left}

Let $A$ be an algebra, not necessarily unital.  Regard $A$ as a right module over itself via right multiplication, and denote this module by $A_{A}$.  We impose the following assumption:
\begin{itemize}
\item[(A1)] \emph{The right module $A_{A}$ is idempotent and non-degenerate.}
\end{itemize}
Then $A$ embeds naturally into the algebra $L(A)=\End(A_{A})$ of left multipliers and we can form the multiplier algebra $M(A) \subseteq L(A)$.  If $A$ has a unit $1_{A}$, then the map $T\mapsto T1_{A}$ identifies $M(A)=L(A)$ with $A$ as an algebra. We denote elements of $A$ by $a,a',b,b',\ldots$

Let $C$ be an algebra, not necessarily unital, with a homomorphism $ s\colon C \to M(A)$ and an anti-homomorphism $t\colon C \to M(A)$ such that $s(C)$ and $t(C)$ commute.  We denote elements of $C$ by $x,x',y,y',\ldots$ We write $\csA$ and $\ctA$ when we regard $A$ as a left or right $C$-module via left multiplication along $s$ or $t$, respectively, that is, $x\cdot a = s(x)a$ and $a\cdot x = t(x)a$. Similarly, we write $\cAs$ and $\cAt$ when we regard $A$ as a right or left $C$-module via right multiplication along $s$ or $t$, respectively.  We make the following assumption:
\begin{itemize}
\item[{(A2)}] \emph{The modules $\csA$ and $\ctA$ are faithful and idempotent.}
\end{itemize}
This condition means that the maps $s$ and $t$ are injective and $s(C)A = A = t(C)A$.  Note that then $C$ is non-degenerate as an algebra. Indeed, if $xC=0$, then $s(x)A=s(x)s(C)A=0$ and hence $x=0$, and if $Cx=0$, then $t(x)A=t(x)t(C)A=t(Cx)A = 0$ and hence $x=0$ again.

We next form the tensor product $\AlA$ of $C$-modules and regard it as a right module over $A\otimes 1$ or $1 \otimes A$ in the obvious way.  We would like the following condition to hold:
\begin{itemize}
\item[{(A3)}] \emph{The space $\AlA$ is non-degenerate as a right module over $A
  \otimes 1$ and  over $1 \otimes A$.}
\end{itemize}

We next list several cases in which this assumption is satisfied, and use the following terminology.  We call a multiplier $E\in M(C^{\op}\otimes C)$ a \emph{left separability multiplier} if for every element $x\in C$, we have
\begin{align*}
E(x^{\op} \otimes 1) = E(1 \otimes
    x) \in C^{\op} \otimes C,
\end{align*}
and the linear map $y^{\op} \otimes z \mapsto yz$ sends this element above to $x$; see also \new \cite[\S 1]{daele:separability}. \old An algebra $D$ is \emph{firm} if the multiplication map $D \underset{D}{\otimes} D \to D$ is an isomorphism, and a module $M$ over an algebra $D$ is \emph{locally projective} if for every finite subset $F \subseteq M$, there exist finitely many morphisms $\upsilon_{i} \in \Hom(M,D)$ and $m_{i} \in \Hom(D,M)$, where $D$ is regarded as a $D$-module in the obvious way, such that
\begin{align*}
  \sum_{i} m_{i}(\upsilon_{i}(m)) = m \quad \text{for all } m\in F;
\end{align*}
see \new \cite[Theorem 2.1]{zimmermann-huisgen}.  \old In this case, $M$ is also \emph{universally torsionless} and a \emph{trace module}, see \new \cite[Theorem 3.2]{garfinkel}\old.  The module $M$ is \emph{projective} if there exist $\upsilon_{i}$ and $m_{i}$ as above, but possibly infinitely many, such that the sum above is finite and equal to $m$ for every $m\in M$.
\begin{lemma} \label{lemma:non-degenerate} 
Assume that (A1), (A2) and  one of the following conditions holds:
  \begin{enumerate}
  \item The right module $A_{A}$ has local units in $A$.
  \item There exists a left separability multiplier $E \in M(C^{\op}\otimes C)$.      
  \item The algebra $C$ is firm and the $C$-modules $\csA$ and $\ctA$
    are locally projective.
  \end{enumerate}
Then condition (A3) is satisfied.
\end{lemma}
\begin{proof}
  (1) Straightforward.

  (2) The assumptions on $E$ imply that the map $j\colon \AlA \to A\otimes A$
  given by $a\otimes b\mapsto (t\otimes s)(E)(a \otimes b)$ is well-defined and that the
  canonical map $A \otimes A \to \AlA$ is a left inverse to $j$. Since $A \otimes A$ is
  non-degenerate as a right module over $A\otimes 1$ and over $1\otimes A$, so is the image $j(\AlA)$ and hence also $ \AlA$.

  (3)  Let $w=\sum_{k} a_{k} \otimes b_{k} \in \AlA$ and assume $w(c\otimes 1)=0$ for all
  $c\in A$.  Choose $\upsilon_{i} \in
    \Hom(\ctA,\Cc)$, $e_{i} \in \Hom(\Cc,\ctA)$ and $ f_{j} \in \Hom(\csA,\cC)$,
    $\omega_{j} \in \Hom(\cC,\csA)$ such that 
 $\sum_{i}e_{i}(\upsilon_{i}(a_{k})) =a_{k}$ and $\sum_{j} f_{j}(\omega_{j}(b_{k}))=b_{k}$
  for all $k$.  Fix $i$ and $j$. Then $\sum_{k}  t(\omega_{j}(b_{k}))a_{k}c=0$ for all $c\in A$, whence
  $\sum_{k} t(\omega_{j}(b_{k}))a_{k} = 0$  by (A1) and hence $\sum_{k}
  \upsilon_{i}(a_{k}) \omega_{j}(b_{k})= 0$. Since $C$ is firm, we can
  conclude $\sum_{k} \upsilon_{i}(a_{k})\otimes \omega_{j}(b_{k})  =0$
  in $C \underset{C}{\otimes} C$.  We apply $e_{i} \otimes f_{j}$, sum
  over $i$ and $j$, and get $w=0$. Therefore, $\AlA$ is
  non-degenerate as a right module over $A\otimes 1$.  A similar
  argument shows  that it is  non-degenerate over $1\otimes A$ as well.
\end{proof}

Let $A,C$ be algebras and $s,t\colon C \to M(A) \subseteq L(A)$ be
maps with commuting images such that (A1)--(A3) hold.   
\begin{remark}
  Before we proceed, let us note that we choose a slightly different notation than in
\cite{daele:relation}, where roles of $s$ and $t$ are switched and $C$ is implicitly
  replaced \new by \old $B=C^{\op}$. With the  present choice,  the space $\AlA$, which carries the target of the comultiplication, is a balanced tensor product of a right module with a left module, whereas in \cite{daele:relation} it was the balanced tensor product of a  left with a right module which  may lead to some confusion. 
\end{remark}

 The \new  left comultiplication on $A$ takes values in \old the subspace
\begin{align*} 
\AltkA \subseteq \End(\AlA)
\end{align*} 
 formed by all endomorphisms $T$ of $\AlA$ satisfying
the following condition: 
\begin{quote}
  \emph{ For every $a,b \in A$, there exist elements
  \begin{align*} T(a \otimes 1) \in \AlA \quad \text{and}
    \quad T(1 \otimes b) \in \AlA
  \end{align*}
  such that $T(a \otimes b) = (T(a \otimes 1))(1 \otimes b) = (T(1
    \otimes b))(a \otimes 1)$.}
\end{quote}
This subspace is a subalgebra and commutes with the right $A\otimes
A$-module action. Note that the elements $T(a\otimes 1)$ and $T(1
\otimes b)$ are uniquely determined thanks to the non-degeneracy
assumption on $\AlA$.

If $A$ has a unit $1_{A}$, then the map $\AltkA \to \AlA$ given by $T \mapsto T(1_{A} \otimes 1_{A})$ identifies $\AltkA$ with the left Takeuchi product, which is the algebra
\begin{align} \label{eq:left-takeuchi}
\new \scA \times \cA  = \old  \left\{ w \in \AlA :
w(t(x) \otimes 1) =  w(1\otimes s(x)) \text{ for all } x\in
C\right\} \subseteq \AlA.
\end{align}
\begin{lemma}
Let $\Delta \colon A \to \AltkA$ be a linear map. Then the linear maps 
\begin{align*}
   \widetilde{T_{\lambda}},\widetilde{T_{\rho}} \colon A \otimes A \to \AlA
\end{align*}
 given by
  \begin{align} \label{eq:delta-tltr} \widetilde{T_{\lambda}}(a\otimes b) &= \Delta(b)(a\otimes 1), &
    \widetilde{T_{\rho}}(a \otimes b) &= \Delta(a)(1 \otimes b),
  \end{align}
for all $a,b\in A$ satisfy
  \begin{align} \label{eq:tltr-welldefined} \widetilde{T_{\lambda}}(t(x)a\otimes b) &=
    \widetilde{T_{\lambda}}(a \otimes b)(1 \otimes s(x)), & \widetilde{T_{\rho}}(a\otimes s(y)b) &=
    \widetilde{T_{\rho}}(a \otimes b)(t(y) \otimes 1)
  \end{align}
  for all $a,b\in A$ and $x,y\in B$ and make the following diagrams commute,
  \begin{gather} \label{dg:tltr-compatible} \xymatrix@R=0pt@C=40pt{ & \AlA \otimes A
      \ar[rd]^{\id \otimes m} & \\ A \otimes A \otimes A \ar[ru]^{\widetilde{T_{\lambda}} \otimes \id}
      \ar[rd]_{\id \otimes \widetilde{T_{\rho}}} & & \AlA, \\ & A \otimes \AlA
      \ar[ru]_{m^{\op} \otimes \id} }
\end{gather}
\begin{gather}\label{dg:tltr-module} \xymatrix@C=25pt@R=15pt{A \otimes A
      \otimes A \ar[r]^(0.45){\id \otimes \widetilde{T_{\lambda}}} \ar[d]_{m^{\op} \oo \id} & A \otimes
      \AlA \ar[d]^{m^{\op} \otimes \id} & A \otimes A \otimes A \ar[r]^(0.45){\widetilde{T_\rho} \otimes
        \id} \ar[d]_{\id \otimes m}
      & \AlA \otimes A \ar[d]^{\id \otimes m} \\
      A\otimes A \ar[r]^{\widetilde{T_{\lambda}}} & \AlA, & A\otimes A \ar[r]^{\widetilde{T_\rho}} & \AlA, }
  \end{gather}
  where $m\colon A\otimes A \to A$ denotes the multiplication, $m^{\op} \colon A\otimes A \to A$ the
  opposite multiplication, and tensor products over $\C$ and over $C$ appear side by side.

  Conversely, every pair of linear maps $(\widetilde{T_{\lambda}},\widetilde{T_{\rho}})$ which make diagram
  \eqref{dg:tltr-compatible} commute and satisfy \eqref{eq:tltr-welldefined} determines a linear
  map $\Delta\colon A \to \AltkA$ by \eqref{eq:delta-tltr}, and each of the maps
  $\Delta,\widetilde{T_{\lambda}},\widetilde{\Tr}$ determines the other two.
\end{lemma}

Before we can list the defining properties of a left comultiplication $\Delta$ and the
corresponding properties of the associated maps  $\widetilde{T_{\lambda}}$ and $\widetilde{\Tr}$, we need to fix some notation.
\begin{notation} \label{notation:left} 
  \begin{enumerate}
  \item We need to consider iterated tensor products of vector
    spaces, of $C$-modules and of $C$-bimodules.  For example, we
    write 
    \begin{align*}
      \ctA \oo A \oo \csA \quad \text{and} \quad \ctA \oo {^{C}A^{C}}
      \oo \csA
    \end{align*}
    for the quotients of $A\otimes A \otimes A$ by the subspaces
    spanned by all elements of the form $t(x)a\otimes b\otimes c -
    a\otimes b\otimes s(x)a$ in the first case, or of the form
    $t(x)a\otimes b\otimes c - a\otimes bt(x)\otimes c$ or $a\otimes
    t(x)b\otimes c - a\otimes b \otimes s(x)c$ in the second case.

  \item We fix an algebra $B$ with an anti-isomorphism $\kappa\colon B\to C$, use this anti-isomorphism to regard $A$ as a $B$-module in various ways, and write
    \begin{align*}
      \btA \oo \bAt \quad \text{and} \quad \bAs \oo \bsA
    \end{align*}
    for the quotients of $A \oo A$ by the subspaces spanned by all elements of the form $s(\kappa(x))a\oo b - a\oo bs(\kappa(x))$ in case of $\btA \oo \bAt$, or $t(\kappa(x))a \oo b - a\oo bt(\kappa(x))$ in case of $\bAs \oo \bsA$. Note that this definition does not depend on the choice of $B$ or $\kappa$. From Section \ref{section:hopf} on, we shall fix a specific $B$ and $\kappa$ and explicitly define the underlying $B$-module structures on $A$, which do depend on this choice.

  \item     Given vector spaces $V$ and $W$, we denote by $\Sigma_{(V,W)}
    \colon V\otimes W \to W \otimes V$ the flip map. In case $V=W=A$, the flip map descends to isomorphisms
    \begin{align*}
      \Sigma_{(\Ac,\cA)} &\colon \Ac \oo \cA \to \btA \oo \bAt, & \Sigma_{(\ctA, \cAt)} &\colon \ctA \oo \cAt \to \Ab \oo \bA.
    \end{align*}
  \item We adopt the usual leg notation for maps on tensor product. For example, we write $(\widetilde{T_{\lambda}})_{13}$ for the composition
    \begin{align*}
      A \otimes A \otimes A \xrightarrow{\id \otimes \Sigma_{(A,A)}} A
      \otimes A \otimes A \xrightarrow{\widetilde{T_{\lambda}} \otimes
        \id} \AlA \otimes A \xrightarrow{\id \otimes
        \Sigma_{(A,A)}} \ctA \otimes A \otimes \csA.
    \end{align*}
  \item The multiplication maps $m\colon A \otimes A \to A$ and $m^{\op} = m\circ \Sigma_{(A,A)} \colon A \otimes A\to A$ \new descend \old to maps
    \begin{align*}
&      \Ac \oo \cA \xrightarrow{m_{C}} A, & & \Ab \oo \bA \xrightarrow{m_{B}} A, &
 & \ctA \oo \cAt \xrightarrow{m_{C}^{\op}} A, & & \btA \oo \bAt \xrightarrow{m_{B}^{\op}} A.
    \end{align*}
  \end{enumerate}
\end{notation}

With this notation at hand,  we can write down the key conditions on $\Delta$, $\widetilde{\Tl}$ and $\widetilde{\Tr}$.
 Note that  \eqref{eq:tltr-welldefined} is equivalent to saying that $\widetilde{\Tl}$ and $\widetilde{\Tr}$ are maps of $C$-modules
\begin{align} \label{eq:tltr-module-1}
\widetilde{\Tl}&\colon  \ctA \oo A \to \ctA \oo {_{C}A_{C}}, &
\widetilde{\Tr} &\colon A \oo \cA \to {^{C}A^{C}} \oo \cA.
\end{align}

\begin{lemma} \label{lemma:tltr} Let $\Delta \colon A \to \AltkA$ and $\widetilde{T_{\lambda}},\widetilde{T_{\rho}}\colon A
\otimes A \to \AlA$ be linear maps related by \eqref{eq:delta-tltr}.
\begin{enumerate}
\item The map $\Delta$ is a homomorphism if and only if one (and then both) of the
  following diagrams commute:
  \begin{align} \label{dg:tltr-multiplicative} \xymatrix@R=15pt{ A \otimes A \otimes A
      \ar[r]^(0.6){\id \otimes m} \ar[d]_{(\widetilde{T_{\lambda}})_{13}}
      & A \otimes A \ar[dd]^{\widetilde{T_{\lambda}}} && A \otimes A
      \otimes A \ar[r]^(0.6){m^{\op} \otimes \id} \ar[d]_{
        (\widetilde{T_{\rho}})_{13}} & A \otimes A \ar[dd]^{\widetilde{T_{\rho}}} \\
      \ctA \otimes A \otimes \csA \ar[d]_{\widetilde{T_{\lambda}} \otimes
        \id} & && \ctA\otimes A \oo \csA
      \ar[d]_{\id \otimes \widetilde{ T_{\rho}}} \\ \ctA \otimes {_{C}A_{C}} \otimes
      \csA \ar[r]_(0.6){\id \otimes m_{C}} & \AlA &&
      \ctA \otimes {^{C}A^{C}} \otimes \csA \ar[r]_(0.6){m^{\op}_{C} \otimes \id}
      & \AlA }
      \end{align}
    \item The map $\Delta$ satisfies
      \begin{align} \label{eq:left-delta-bimodule}
        \Delta(s(y)t(x)as(y')t(x')) = (s(y) \otimes
        t(x))\Delta(a)(s(y') \otimes t(x'))
      \end{align} 
if and only if one (and then both) of the following conditions hold:
\begin{align} \label{eq:tltr-bimodule}
  \begin{aligned}
    \widetilde{T_{\lambda}}(a \otimes t(x)s(y)bt(x')s(y')) &= (s(y) \otimes t(x))\widetilde{T_{\lambda}}(s(y')a
    \otimes b)(1 \otimes t(x')), \\ 
    \widetilde{T_{\rho}}(t(x)s(y)at(x')s(y') \otimes b) &= (s(y) \otimes t(x))\widetilde{T_{\rho}}(a \otimes
    t(x')b)(s(y') \otimes 1).
  \end{aligned}
      \end{align} 
      If these conditions hold, then $\widetilde{T_{\lambda}}$ and
      $\widetilde{T_{\rho}}$ descend to maps
\begin{align*} T_{\lambda}&\colon \btA \otimes \bAt \to \AlA,
  & T_{\rho}&\colon \bAs \otimes \bsA \to \AlA.
\end{align*} 
\item Assume that the conditions in (2) hold. Then $\Delta$ is coassociative in the sense
that
    \begin{align} \label{eq:delta-coassociative} (\Delta \otimes \iota)(\Delta(b)(1 \otimes c)) (a\otimes 1 \otimes 1) =
(\iota\otimes \Delta)(\Delta(b)(a \otimes 1))(1 \otimes 1 \otimes c)
    \end{align} 
if and only if the following diagram commutes:
    \begin{align} \label{dg:tltr-coassociative} \xymatrix@R=18pt{ 
        A \otimes A \otimes A \ar[r]^(0.45){\id \otimes \widetilde{\Tr}} \ar[d]_{\widetilde{\Tl} \otimes \id} &
A \otimes \AlA \ar[d]^{\widetilde{\Tl} \otimes \id} \\  \AlA \otimes A \ar[r]^(0.45){\id \otimes \widetilde{
  \Tr}} & \AlAlA.  }
    \end{align}
  \end{enumerate}
\end{lemma}
Note that  \eqref{eq:tltr-bimodule}  implies that $\widetilde{\Tl}$ and $\widetilde{\Tr}$ are maps of  $C$-bimodules 
\begin{align} \label{eq:tltr-bimodule-1}
  \widetilde{\Tl}&\colon A \oo {_{C}A^{C}} \to {_{C}A^{C}} \oo {_{C}A^{C}}, & 
  \widetilde{\Tr}&\colon {_{C}A^{C}} \oo A \to {_{C}A^{C}} \oo {_{C}A^{C}}.
\end{align}

If $A$ is unital and $\AltkA$ is identified with the left Takeuchi product
 \eqref{eq:left-takeuchi}, then equations \eqref{eq:left-delta-bimodule} and
 \eqref{eq:delta-coassociative} reduce to the conditions
 \begin{align} \label{eq:delta-unital}
   \Delta(t(x)s(y)) &= s(y) \oo t(x),  &
   (\Delta \otimes \id)\circ \Delta=   (\id \otimes \Delta) \circ \Delta.
 \end{align}

If (A1)--(A3) and
\eqref{eq:left-delta-bimodule}, \eqref{eq:delta-coassociative} hold, we
call $(A,C,s,t,\Delta)$ a left multiplier bialgebroid:
\begin{definition}  \label{definition:left-algebroid}
  A \emph{left multiplier bialgebroid} is a tuple $\mathcal{A}=(A,C,s,t,\Delta)$ consisting of
  \begin{enumerate}
  \item algebras $A$ and $C$ such that the right $A$-module $A_{A}$ is
    non-degenerate and idempotent;
  \item a homomorphism $s\colon C \to M(A)$ and an anti-homo\-morphism
    $t \colon C \to M(A)$ such that the images of $s$ and $t$ commute,
    the $C$-modules $\csA$ and $\ctA$ are faithful and idempotent, and
    $\AlA$ is non-degenerate as a right module over $A
    \otimes 1$ and over $1 \otimes A$;
  \item a homomorphism $\Delta \colon A \to \AltkA$, called the
    \emph{left comultiplication}, which satisfies the $C$-bilinearity
    condition \eqref{eq:left-delta-bimodule} and the coassociativity
    condition \eqref{eq:delta-coassociative}.
\end{enumerate}
We call the maps $T_{\lambda}$ and $T_{\rho}$ defined above the  \emph{canonical
  maps} associated to $\mathcal{A}$. 
We call a left multiplier bialgebroid as above \emph{unital} if the algebras $A,C$ and the maps $s,t,\Delta$ are unital.
\end{definition}

Given a left multiplier bialgebroid,
one can reverse the comultiplication as follows. 
\begin{proposition} \label{prop:left-co-opposite} Let
  $\mathcal{A}=(A,C,s,t,\Delta)$ be a left multiplier bialgebroid with
  associated maps $(\widetilde{\Tl},\widetilde{\Tr})$.  Regard $s$
  as an anti-homomorphism and $t$ as a homomorphism from $C^{\op}$ to
  $M(A)$. Write $A^{C^{\op}}$ and $_{C^{\op}}A$ for $A$, regarded as a $C^{\op}$-module via $a\cdot y^{\op}:= s(y)a$ and $y^{\op}\cdot a := t(y)a$, where $y\in C$ and $a\in A$. Then the  flip map $\Sigma_{(A,A)}$ on $A\otimes A$ descends to an isomorphism 
  $\Sigma_{(\ctA,\csA)}$ from $\AlA$ to ${A^{C^{\op}}} \otimes {_{C^{\op}}A}$,
  there exists a well-defined homomorphism
  \begin{align*}
\Delta^{\co} \colon A \to {A^{C^{\op}}} \overline{\times} {_{C^{\op}}A}, \quad    \Delta^{\co}(a)(b\otimes c) = \Sigma_{(\ctA,\csA)}(\Delta(a)(c\otimes b)),
  \end{align*}
 and
$\mathcal{A}^{\co}:=(C^{\op},A,t,s,\Delta^{\co})$
is a left multiplier bialgebroid with associated maps
 \begin{align*}
   \widetilde{\Tl}^{\co} &= \Sigma_{(\ctA,\csA)} \circ \widetilde{\Tr} \circ \Sigma_{(A,A)}, & 
   \widetilde{\Tr}^{\co} &= \Sigma_{(\ctA,\csA)} \circ \widetilde{\Tl} \circ \Sigma_{(A,A)}.
 \end{align*}
\end{proposition}
The proof is straightforward and therefore omitted.

The canonical maps satisfy  pentagonal relations:
\begin{proposition} \label{prop:tltr-pentagon} Let $(A,C,s,t,\Delta)$ be a left multiplier
  bialgebroid. If $\AlAlA$ is non-degenerate as a right
    module over $A \otimes 1\otimes 1$ and over $1\otimes 1\otimes A$, then the
    following diagrams commute: 
  \begin{gather*}
\xymatrix@C=40pt@R=15pt{ 
A \oo A \oo A \ar[r]^(0.45){(\widetilde{\Tl})_{12}(\widetilde{\Tl})_{23}} \ar[d]_{(\widetilde{\Tl})_{12}} & \AlAlA, \\ \ar[r]^{(\widetilde{\Tl})_{13}} 
\AlA \oo A &
A^{C} \otimes  A^{B} \oo {^{B}_{C}A} \ar[u]_{(\Tl)_{23}} }   \qquad
\xymatrix@C=40pt@R=15pt{ 
A \oo A \oo A \ar[d]_{(\widetilde{\Tr})_{23}} \ar[r]^(0.45){(\widetilde{\Tr})_{23}(\widetilde{\Tr})_{12}}  & \AlAlA.  \\
A \oo \AlA \ar[r]^{(\widetilde{\Tr})_{13}} &
A^{C}_{B} \oo \bA \oo  \cA\ar[u]_{({\Tr})_{12}} }
  \end{gather*}
\end{proposition}
\begin{proof}   
The pentagonal relation for $\widetilde{\Tr}$
follows from commutativity of the diagram
  \begin{align*} \xymatrix@C=20pt@R=18pt{  
&A \otimes  A \otimes A \otimes A \ar[r]^{(\widetilde{\Tr})_{23}}
\ar `l/0pt[l] `d[lddd] [dddr]_{(\widetilde{\Tr})_{24}(\widetilde{\Tr})_{34}} \ar[d]_{(\widetilde{\Tl})_{12}}& 
A \otimes \AlA \oo A
\ar[r]^{(\widetilde{\Tr})_{34}} \ar[d]^{m^{\op} \otimes \id \otimes \id} & 
A \oo \AlAlA
\ar[d]^{m^{\op}
\otimes \id \otimes \id} & \\ & 
\AlA \oo A \oo A
 \ar[r]^(0.55){\id \otimes m \otimes \id}
\ar[d]_(0.4){(\widetilde{\Tr})_{34}} & \AlA \oo A \ar[r]^{(\widetilde{\Tr})_{23}} & \AlAlA, \\ & 
\AlA \oo \AlA  \ar[r]^{(\widetilde{\Tr})_{24}}& 
\ctA \oo  {_{C}A_{B}^{C}} \oo \bA \oo \cA
\ar[ru]^{\id \otimes m_{B} \otimes \id} & 
A \oo \AlAlA \ar[u]_{m^{\op} \otimes \id \otimes \id} \\ & &
A \oo A^{C}_{B} \oo \bA \oo \cA
\ar[u]_{(\widetilde{\Tl})_{12}} \ar[ru]_{({\Tr})_{23}} & }
  \end{align*} 
and the  pentagonal relation for $\Tl$ can be concluded similarly. 
\end{proof}
\begin{remark}
  Similar arguments as those used in the proof of Lemma \ref{lemma:non-degenerate} show that the
  assumption on the module $\AlAlA$ holds if condition (1) or (2)
  of Lemma \ref{lemma:non-degenerate} is satisfied, or if the algebra $C$ is firm and the modules
  $\csA$ and $\ctA$ are locally projective. In the latter case, one uses the fact that then $\csA\otimes\ctA$ is
  projective with respect to the  $C$-modules structure given by $x\cdot (a\otimes b) = a\otimes
  t(x)b$ and $(a\otimes b) \cdot x = s(x)a\otimes b$, respectively.
\end{remark}

\new
Throughout this article, we shall mainly use the canonical maps instead of the comultiplication itself. To make some formulas and calculations more accessible, we also write them out in a generalized Sweedler notation,  which is more
intuitive but a bit difficult to make precise. We shall not attempt to formalize it and note that for
every expression involving this notation, one needs to check whether it is well-defined.  In the
context of multiplier Hopf algebras, the correct usage of this notation is explained in
\cite{daele:0,daele:tools}.  With these words of warning, given a left multiplier bialgebroid
$\mathcal{A}=(A,C,s,t,\Delta)$, we write
\begin{align*} 
\Delta(a) &=  a_{(1)} \otimes a_{(2)} \in \End(\AlA)
 \end{align*}
for all $a\in A$,
 where the right hand sides
are  purely formal expressions.  For example, we then have
\begin{gather} \label{eq:sweedler-product}
(ab)_{(1)} \otimes (ab)_{(2)} = a_{(1)}b_{(1)} \otimes a_{(2)}b_{(2)},  \\  \label{eq:sweedler-module}
\begin{aligned}
 a_{(1)}
  \otimes a_{(2)} s(z) &= a_{(1)} t(z) \otimes a_{(2)},
\end{aligned} \\ \label{eq:sweedler-canonical}
\begin{aligned} \Tl (a\otimes b) &=  b_{(1)}a \otimes
b_{(2)}, & \Tr(a\otimes b) &=   a_{(1)} \otimes a_{(2)}b,
   \end{aligned} \\
  \Delta^{\co}(a) = a_{(2)} \otimes a_{(1)}
\end{gather}
for all $a,b\in A$ and $z\in C$,
and the pentagonal relation for $\widetilde{\Tr}$ takes the form
\begin{align*}
  a_{(1)} \otimes (a_{(2)}b)_{(1)} \otimes (a_{(2)}b)_{(2)} c = (a_{(1)})_{(1)} \otimes (a_{(1)})_{(2)}b_{(1)} \otimes a_{(2)}b_{(2)}c.
\end{align*}
\old
\section{Counits for left multiplier bialgebroids} \label{section:left-counits}

\new

 We next discuss counits of left multiplier bialgebroids, and prove uniqueness, some multiplicativity, and
existence in the case where the canonical maps are surjective or bijective, respectively. 

Let us fix some notation. \old
 Given a left multiplier bialgebroid $(A,C,s,t,\Delta)$ and morphisms
 $\phi \in \Hom(\ctA,\Cc)$ and $\psi \in \Hom(\csA,\cC)$, we can form slice maps
\begin{align*}  \phi \new \odot \old \id &\colon \AlA \to A, \ a \otimes b \mapsto s(\phi(a))b, & \id \new \odot \old \psi &\colon \AlA \to A, \ a
  \otimes b \mapsto t(\psi(b))a.
\end{align*}
\begin{definition} \label{df:left-counit} A \emph{left counit} for a left multiplier bialgebroid $(A,C,s,t,\Delta)$ is a
  map $\varepsilon \colon A \to C$  that satisfies
  \begin{align}
    \label{eq:left-counit-bimodule}
    \varepsilon(s(y)a ) &= y\varepsilon(a) &&\text{and}& \varepsilon(t(x)a) &=
    \varepsilon(a)x && \text{for all } a\in A, x,y\in C,
  \end{align}
 that is, $\varepsilon \in \Hom(\csA,\cC) \cap
  \Hom(\ctA,\Cc)$,\sindex{epsilon}{$\varepsilon$}{left/right counit} and
  \begin{align}
        \label{eq:left-counit}
        (\varepsilon \new \odot \old\id)( T_{\rho}(a\otimes b)) &= ab && \text{and} &   (\id
        \new \odot \old \varepsilon)( T_{\lambda}(a \otimes b))&=ba && \text{for
        all } a,b\in A.
  \end{align}
\end{definition}
\new \begin{remark}
  \begin{enumerate}
  \item In Sweedler notation, \eqref{eq:left-counit} takes the form
    \begin{align} \label{eq:sweedler-left-counit}
      s(\varepsilon(a_{(1)}))a_{(2)}b = ab  \quad \text{and} \quad
      t(\varepsilon(b_{(2)}))b_{(1)}a = ba \quad \text{for all } a,b\in A.
    \end{align} 
  \item \old Note that left counits for a left multiplier bialgebroid $\mathcal{A}$
and left counits for its co-opposite $\mathcal{A}^{\co}$ introduced in
Proposition \ref{prop:left-co-opposite} coincide up to the canonical
linear identification $C\to C^{\op}$.
\new  \end{enumerate}
\end{remark}\old
 
If $A$ has a unit $1_{A}$, we can identify $\AltkA$ with the
left Takeuchi product \eqref{eq:left-takeuchi}, and
then commutativity of the diagrams above is equivalent to the
equations
\begin{align} \label{eq:left-counit-unital-1} (\varepsilon \new \odot \old \id) \circ \Delta &= \id_{A}= (\id \new \odot \old \varepsilon) \circ \Delta.
\end{align}
From these equations, one can easily deduce that a left counit, if it exists, is unique. If it also is multiplicative in the sense that 
\begin{align} \label{eq:left-counit-unital-2}
  \varepsilon(ab) = \varepsilon(as(\varepsilon(b))) = \varepsilon(at(\varepsilon(b)))
\end{align}
for all $a,b\in A$, then we obtain a left bialgebroid in \new  the well-known sense as described in, for example,  in \cite{boehm:hopf}, \cite{lu:hopf}:  \old
\begin{proposition} \label{prop:left-unital} Let $(A,C,s,t,\Delta)$ be a  unital left multiplier bialgebroid with a left counit $\varepsilon$ that is unital and satisfies \eqref{eq:left-counit-unital-2}. Then we can  regard $A$ as an $C\otimes C^{\op}$-ring and as a $C$-bimodule via $y \cdot a \cdot x = t(x)s(y)a$ for all $x,y\in C,a\in A$, and  $\Delta$ as a homomorphism from $A$ to $\AltkA \cong A {_{C}\times} A$. The tuple $(A,\Delta,\varepsilon)$ is a $C$-coring and, together with the $C\otimes C^{\op}$-ring structure on $A$, forms a left bialgebroid.  Conversely, every left bialgebroid arises this way from a unital left multiplier bialgebroid with a \new unital \old  left counit satisfying \eqref{eq:left-counit-unital-2}.
\end{proposition}
\begin{proof} 
   \new Straightforward. \old
\end{proof}

In the non-unital case, we can prove uniqueness and multiplicativity of left counits only under additional assumptions which are analogues of the conditions in  \cite[Definition 1.4]{daele:weakmult}.
\begin{definition} \label{df:full} 
We call a left multiplier bialgebroid $(A,C,s,t,\Delta)$ \emph{left-full} if $A$ is   equal to the linear span of elements of the form  $(\id \odot \psi)(\widetilde{T_{\rho}}(a \otimes b))$,  where $\psi \in \Hom(\csA,\cC)$ and $a,b\in A$,   \emph{right-full} if $A$ is equal to the linear span of elements of the form   $(\phi \odot \id)(\widetilde{T_{\lambda}}(a\otimes b))$,  where $\phi
  \in \Hom(\ctA,\Cc)$ and $a,b\in A$,  and \emph{full} if it is both left-full and right-full.
\end{definition}
In the unital case, \eqref{eq:left-counit-unital-1} shows that existence of a left counit implies
fullness.  In general, we only know the following:
\begin{remark} \label{rmk:full}
 If a left multiplier bialgebroid $(A,C,s,t,\Delta)$ has a left counit $\varepsilon$ and its
  canonical map $T_{\lambda}$ (or $T_{\rho}$) is surjective, then it is left-full (resp.\
  right-full). To see this, take $\phi$ (or $\psi$) above to be equal to $\varepsilon$ and use the
  relation $AA=A$.
\end{remark}
If the left multiplier bialgebroid is full, then  the left counit, if it exists, is unique:
\begin{proposition} \label{prop:left-counit-props}
  Let $\mathcal{A}=(A,C,s,t,\Delta)$ be a left multiplier bialgebroid
  with a left counit $\varepsilon$.
  \begin{enumerate}
  \item If $\mathcal{A}$ is left-full or right-full,  then the left counit is unique.
  \item If the canonical map $T_{\lambda}$ (or $T_{\rho}$) is surjective, then for
    all $a,b\in A$,
    \begin{align} \label{eq:left-counit-multiplicative} 
      \varepsilon(ab) &= \varepsilon(as(\varepsilon(b))) \quad (\text{or } \varepsilon(ab) =\varepsilon(at(\varepsilon(b))),
      \text{ respectively}).
    \end{align}
  \end{enumerate}
\end{proposition}
\begin{proof}
  (1) Assume that $\mathcal{A}$ is left-full and that $\varepsilon$ is a left counit. Let $a,b\in A, \psi \in
  \Hom(\csA,\cC)$ and write $\widetilde{\Tr}(a\otimes b) = \sum_{i} c_{i} \otimes d_{i}$ with $c_{i},d_{i} \in A$. Then
  \eqref{eq:left-counit} shows that $\sum_{i} s(\varepsilon(c_{i}))d_{i} = ab$ and hence
  \begin{align*} \varepsilon\Big(\sum_{i} t(\psi(d_{i}))c_{i}\Big) &= \sum_{i} \varepsilon(c_{i}) \psi(d_{i}) = \sum_{i}
    \psi(s(\varepsilon(c_{i}))d_{i}) = \psi(ab).
  \end{align*}
  But since $(A,C,s,t,\Delta)$ is assumed to be left-full, elements of the form $\sum_{i} t(\psi(d_{i}))c_{i}$
  span $A$. If $(A,C,s,t,\Delta)$ is right-full, a similar argument applies.

  (2) Assume that $\Tr$ is surjective and consider the following diagram: 
  \begin{align*} \xymatrix@C=35pt@R=18pt{A \otimes A \otimes A
      \ar[r]^(0.4){\id \otimes \widetilde{T_{\rho}}} \ar[rd]_{\id \otimes
      m} \ar[dd]_{m \otimes
        \id} & A \otimes \AlA \ar[d]^{\id \otimes (\varepsilon
 \new        \odot \old \id)} \ar[r]^{(\widetilde{T_{\rho}})_{13}} & A^{C}_{B} \oo {_{B}A} \oo \cA \ar
      `r/0pt[r] [rdd]^{m_{B}
        \otimes \id} \ar[d]^{\id \otimes (\varepsilon \new \odot \old \id)} & \\
      & A \otimes A \ar[d]_{m}  \ar[r]^{\widetilde{T_{\rho}}} & \AlA
      \ar[d]^{\varepsilon \new \odot \old \id} & \\  A\otimes A
      \ar `d/0pt[d] `r[rrrd]^{\widetilde{T_{\rho}}}  [rrr] \ar[r]^{m} &A \ar@{=}[r] &A &  \AlA \ar[l]_{\varepsilon
        \new \odot \old \id}  \\ &&&& }
  \end{align*} 
  The outer cell commutes by \eqref{dg:tltr-module}, and all other cells except for the right one commute as
  well. Since $\widetilde{\Tr}$ is surjective, we can conclude that the right cell must commute. Therefore,
$\varepsilon(ab) =\varepsilon(at(\varepsilon(b)))$ for all $a,b\in A$. If $\Tl$ is surjective,  a similar
  argument applies.
\end{proof}
\new \begin{remark}
 In Sweedler notation, the commutative diagram above amounts to the  calculation
\begin{align*}
s(\epsilon(a_{(1)}b_{(1)}))a_{(2)}b_{(2)}c &=
  s(\varepsilon((ab)_{(1)})) (ab)_{(2)}c & &\text{(by  \eqref{eq:sweedler-product})}   \\ & =  (ab)c  & &\text{(by  \eqref{eq:sweedler-left-counit})}  \\ 
&= s(\varepsilon(a_{(1)}))a_{(2)}s(\epsilon(b_{(1)}))b_{(2)}c & &\text{(by  \eqref{eq:sweedler-left-counit})}  \\
&=
 s(\varepsilon(a_{(1)})t(\epsilon(b_{(1)})))a_{(2)}b_{(2)}c & &\text{(by  \eqref{eq:sweedler-module})}, 
\end{align*}
which, thanks to surjectivity of $\Tr$, implies $\varepsilon(a'b') = \varepsilon(a't(b'))$ for all $a',b'\in A$.
\end{remark} \old

 \new Let $(A,C,s,t,\Delta)$ be a left multiplier bialgebroid. We shall prove existence of a left counit provided that the canonical maps $\Tr$ and $\Tl$ are bijective and a further technical condition holds. This condition involves the left
 ideal ${_{C}I} \subseteq C$ and the right ideal ${I^{C}} \subseteq C$ given by
\begin{align*}  {_{C}I} &:= \lspan\{ \psi(a) : \psi
  \in \Hom(\csA,\cC), a \in A\}, \\
 {I^{C}} &:= \lspan\{ \phi(a) : \phi \in \Hom(\ctA,\Cc), a \in A\}.
\end{align*} \old
Recall that a two-sided ideal $I$ in $C$ is \emph{essential} if $yI\neq 0$  and $Iy\neq 0$ whenever $y\in C$ and $y\neq 0$.
\begin{lemma} \label{lm:counit-isit} Let $(A,C,s,t,\Delta)$ be a left multiplier bialgebroid.
  \begin{enumerate}
  \item   Let $(A,C,s,t,\Delta)$ be a left multiplier bialgebroid with a left counit $\varepsilon$. Then  the
  image $C_{0}:=\varepsilon(A)$ is an idempotent, essential two-sided ideal  in $C$, \new contained in ${I^{C}} \cap {_{C}I}$, \old and $s(C_{0})A=A=t(C_{0})A$.
  \item \new If $s({I^{C}})A=A=t({_{C}I})A$, then ${I^{C}}={I^{C}}\cdot {_{C}I}={_{C}I}$ is an idempotent, essential two-sided ideal in $C$. \old
  \end{enumerate}
\end{lemma}
\begin{proof}
  \new (1) \old Equations \eqref{eq:left-counit-bimodule} and \eqref{eq:left-counit} imply that $C_{0}$ is a two-sided ideal, \new contained in ${_{C}I} \cap {I^{C}}$, \old and that $s(C_{0})A=AA=t(C_{0})A$. But $AA=A$ by (A1).  Applying \eqref{eq:left-counit-bimodule}, we conclude that $C_{0}C_{0}=C_{0}$.  The relations $s(C_{0})A=A=t(C_{0})A$ and injectivity of $s$ and $t$ imply that the ideal $C_{0}$ is essential.

\new
(2) Applying elements of $\Hom(\csA,\cC)$ or $\Hom(\ctA,\Cc)$ to the assumed equality, we find ${_{C}I}={I^{C}}\cdot {_{C}I} = {I^{C}}$. If $z\in C$ is non-zero, then, using injectivity of $s$ and $t$, we can conclude that $s(zy)$ and $t(xz)$ are non-zero for some $y\in {I^{C}}$ and $x\in {_{C}I}$. \old
\end{proof}
\new If $s({I^{C}})A=A=t({_{C}I})A$, then we can assume ${I^{C}}=C={_{C}I}$ without much loss of generality: \old
\begin{lemma} Let $\mathcal{A}=(A,C,s,t,\Delta)$ be a left multiplier bialgebroid with a two-sided ideal $C_{0}\subseteq
  C$ such that $s(C_{0})A=A=t(C_{0})A$.  Denote by $s_{0}$ and $t_{0}$ the restrictions of $s$ and $t$,
  respectively, to $C_{0}$. Then the natural map $A^{C_{0}} \otimes {_{C_{0}}A} \to \AlA$ is an
  isomorphism. Denote by $\Delta_{0}$ the composition of $\Delta$ with the induced isomorphism $\AltkA
    \to {A^{C_{0}}} \overline{\times} {_{C_{0}}A}$. Then  $\mathcal{A}_{0}:=(A,C_{0},s_{0},t_{0},\Delta_{0})$ is a left multiplier
    bialgebroid, and every left counit for $\mathcal{A}$ takes values in $C_{0}$ and is a left counit for $\mathcal{A}_{0}$.
\end{lemma}
\begin{proof}
  We first show that the natural map $A^{C_{0}} \otimes {_{C_{0}}A} \to \AlA$ is an isomorphism.
  Given $a,b\in A$ and $x\in C$, we can write $a=\sum_{i} t(x_{i})a_{i}$ with $x_{i} \in C_{0}$ and $a_{i} \in
  A$, and then $t(x)a \otimes b - a \otimes s(x)b$ is equal to
  \begin{align*}
    \sum_{i} \left(t(xx_{i})a_{i} \otimes b - a_{i} \otimes s(xx_{i})b + a_{i} \otimes s(x_{i})s(x)b - t(x_{i})a_{i} \otimes
      s(x)b\right)
  \end{align*}
  and therefore lies in the space spanned by all elements of the form $t(x')a'\otimes b' - a'\otimes
  s(x')b'$, where $a',b'\in A$ and $x'\in C_{0}$.  The first assertion follows.

  It follows immediately that $\mathcal{A}_{0}$ is a left multiplier bialgebroid. 

  If $\varepsilon$ is a left counit for $\mathcal{A}$, then the assumption $A=s(C_{0})A$ and
  \eqref{eq:left-counit-bimodule} imply that $\varepsilon(A)$ is contained in $C_{0}$, and clearly,
  $\varepsilon$ also is a left counit for $\mathcal{A}_{0}$.
\end{proof}
\new
Now, we can prove the existence result: \old
\begin{proposition} \label{prop:left-counit-existence} Let  $(A,C,s,t,\Delta)$ be a left multiplier bialgebroid. If its  canonical maps $(\Tl,\Tr)$ are bijective and  \new $s({I^{C}})A=A=t({_{C}I})A$,\old then it has a unique left counit.
\end{proposition}
\begin{proof}
  Suppose that the assumptions hold.   Then uniqueness of a left counit follows from \ref{prop:left-counit-props}  (1).  To prove existence, consider the linear maps
  \begin{align*} \Es  &:= m_{B} \circ \Rr \colon \AlA \to A, & \Et &:= m_{B}^{\op} \circ \Rl \colon \AlA \to A,
  \end{align*}
Since  $ \Es(a \otimes bc) = \Es(a \otimes b)c$ for all $a,b,c\in A$ by \eqref{dg:tltr-module},
  the formula $\es(a)b:=\Es(a \otimes b)$ defines a map $\es \colon A \to L(A)$. By definition and by
  \eqref{eq:tltr-bimodule},
  \begin{align} \label{eq:et-module} \es(t(z)s(y)a)b &= (\mult \circ \Rr)(s(y)a \otimes s(z)b) = s(y) \es(a)s(z)b
  \end{align} 
  for all $y,z \in C$ and $a,b \in A$.  Likewise, the formula $\et(a)b:=\Et(b \otimes a)$ defines a map $\et
  \colon A \to L(A)$ satisfying
  \begin{align} \label{eq:es-module} \et(t(x)s(z)a)b = t(x)\et(a)t(z)b
  \end{align}
for all $x,z \in B$ and $a,b \in A$.

  Consider the following diagram:
  \begin{align*} \xymatrix@C=40pt@R=10pt{ & \AlAlA \ar[rd]^{\id\otimes \Rr} \ar[ld]_{\Rl \otimes \id} & \\ \ATAlA
      \ar[r]^(0.6){m_{B}^{\op} \otimes \id} \ar[rd]_{\id \otimes \Rr} &
      \AlA & \AlAsA \ar[l]_(0.6){\id \otimes m_{B}} \ar[ld]^{\Rl
        \otimes \id} \\ & \ATAsA. }
  \end{align*} 
The outer square and the lower cell commute by \eqref{dg:tltr-compatible} and \eqref{dg:tltr-coassociative}. Hence, the upper
cell commutes, showing that for all $a,b,c \in A$,
  \begin{align} \label{eq:es-et} \et(b)a \otimes c = \Et(a \otimes b) \otimes c &=a \otimes \Es(b \otimes c) = a \otimes \es(b)c \text{ in }
    \AlA.
  \end{align}
  Applying $\id \otimes \psi$ or $\phi \otimes \id$ with $\phi \in \Hom(\ctA,\Cc)$ and $\psi \in \Hom(\csA,\cC)$, we
  obtain
  \begin{align*} t(\psi(c))\et(b)a &= t(\psi(\es(b)c))a &&\text{and} & s(\phi(\et(b)a))c &= s(\phi(a))\es(b)c.
  \end{align*} 
Let us focus on the first equation. Since $a \in A$ is arbitrary, we can
  conclude $t(\psi(c))\et(b)=t(\psi(\es(b)c))$ for all $b,c\in A$ and hence
   $t({_{C}I})\et(A) \subseteq t({_{C}I})$. Using the assumption and equation
  \eqref{eq:es-module}, we conclude $\et(A)=\et(t({_{C}I})A) = t({_{C}I})\et(A) \subseteq
  t({_{C}I})$.  A similar argument applied to the second equation shows that $\es(A) \subseteq
  s({I^{C}})$. In particular, we get
  \begin{align*} (t^{-1}\circ \et(b))\psi(c)=\psi(\es(b)c) = (s^{-1}\circ \es(b))\psi(c) \quad \text{for
      all } b,c\in A, \psi\in \Hom(\cA,\cC).
  \end{align*}
  Using Lemma \ref{lm:counit-isit}, we can conclude that $s^{-1}\circ \es=t^{-1}\circ \et$, and this map is  a left counit by construction.
\end{proof}

\section{Right multiplier bialgebroids}
\label{section:right}
The definitions and results of sections \S\ref{section:left} and \S\ref{section:left-counits} have
natural right-handed analogues which are briefly summarized below and will be needed for the
definition of multiplier Hopf algebroids in section \S\ref{section:hopf}. For proofs, explanations
and comments, we refer to the corresponding left-handed versions.

Let $A$ be an algebra, not necessarily unital.  We write $_{A}A$ when we regard $A$ as a left
$A$-module, and assume that this module is non-degenerate and idempotent. Denote by
$R(A)=\End({_{A}A})^{\op}$ the algebra of right multipliers of $A$, and write the application of a
$T\in R(A)$ to an $a\in A$ as $aT$, so that $a(TS)=(aT)S$ for all $a\in A$ and $S,T \in
R(A)$. Then $A$ embeds into $R(A)$ and we can form the multiplier algebra $M(A) \subseteq R(A)$. If
$A$ is unital, then the map $M(A) \to A$ given by $T \mapsto 1_{A}T$ is an isomorphism.

Let $B$ be an algebra with a homomorphism $s\colon B\to M(A) \subseteq R(A)$ and an
anti-homomorphism $t\colon B \to M(A) \subseteq R(A)$ such that the images of $s$ and $t$ commute.
We write $\bAs$ and $\bAt$ if we regard $A$ as a right or left $B$-module such that $a\cdot y = as(y)$
or $y\cdot a = at(y)$ for all $a\in A$ and $y\in B$. We similarly write $\bsA$ and $\btA$ when we use multiplication on
the left hand side instead of the right hand side.

Assume that  the tensor product $\ArA$ is non-degenerate as a left module over $A\otimes 1$ and over $1\otimes A$, respectively. We consider the opposite algebra  $\End(\ArA)^{\op}$ and write $(a \otimes b)T$ for the
image of an element $a\otimes b$ under an element $T \in \End(\ArA)^{\op}$, so that $(a\otimes b)(ST) = ((a\otimes b) S)T$ for all
$a,b\in A$ and $S,T\in \End(\ArA)^{\op}$.
Denote by
\begin{align*}
  \ArtkA \subseteq \End(\ArA)^{\op}
\end{align*}
the subspace formed by all endomorphisms $T$ such that for all $a,b\in
A$, there exist elements $(a\otimes 1)T \in \ArA$ and $(1
\otimes b)T \in \ArA$ such that
\begin{align*}
  (a\otimes b)T = (1\otimes b)((a\otimes 1)T) = (a\otimes 1)((1 \otimes b)T).
\end{align*}
This subspace is a subalgebra.  If $A$ has a unit $1_{A}$, then the
map $T \mapsto (1_{A} \otimes 1_{A})T$ identifies this algebra with the 
right Takeuchi product
\begin{align} \label{eq:right-takeuchi}
\new \Ab \times \Asb = \old  \left\{ w \in \ArA :  (s(y) \otimes 1)w  = (1 \otimes t(y))w
    \text{ for all  } y\in B\right\} \subseteq \ArA.
\end{align}
\begin{definition}
  A \emph{right multiplier bialgebroid} is a tuple $(A,B,s,t,\Delta)$ consisting of
  \begin{enumerate}
  \item algebras $A$ and $B$ such that the left $A$-module $_{A}A$ is non-degenerate and idempotent;
  \item a homomorphism $s\colon B \to M(A) \subseteq R(A)$ and an anti-homo\-morphism $t \colon B \to M(A) \subseteq R(A)$ such that the images of $s$ and $t$ commute, the $B$-modules $\bAs$ and $\bAt$ are faithful and idempotent, and $\ArA$ is non-degenerate as a left module over $A \otimes 1$ and over $1 \otimes A$;
\item a homomorphism $\Delta \colon A \to \ArtkA$, called the \emph{right
    comultiplication},  satisfying
  \begin{align} \label{eq:right-delta-bimodule}
    \Delta(t(y)s(x)at(y')s(x')) &= (t(y) \otimes s(x))\Delta(a) (t(y') \otimes s(x')), \\
    (a \otimes 1 \otimes 1) ((\Delta \otimes \id)((1 \otimes
    c)\Delta(b))) &= (1 
\otimes 1 \otimes c)((\id \otimes \Delta)((a \otimes 1)\Delta(b))) \label{eq:right-delta-co-associative}
  \end{align}
for all $a,b,c\in A$ and $x,y\in B$.
\end{enumerate}
We call a right multiplier bialgebroid as above unital if the algebras $A,B$ and the maps $s,t,\Delta$ are unital.
\end{definition}

Let $(A,B,s,t,\Delta)$ be a right multiplier bialgebroid. Then the linear maps
\begin{align*}
  \widetilde{\lT} & \colon A \otimes A \to \ArA, \quad a\otimes b\mapsto (a\otimes 1)\Delta(b), \\
\widetilde{\rT}& \colon A \otimes A \to \ArA, \quad a\otimes b\mapsto (1\otimes b)\Delta(a),
\end{align*}
satisfy the following analogues of relations \eqref{eq:tltr-welldefined} and \eqref{eq:tltr-bimodule},
\begin{align} \label{eq:ltrt-bimodule}
  \begin{aligned}
    \widetilde{\lT}(as(z) \otimes s(x)bt(y')s(x')) &= (1 \otimes
    t(z)s(x))\widetilde{\lT}(a\otimes b)(t(y') \oo s(x')), \\
    \widetilde{\rT}(t(y)at(y')s(x') \otimes bt(z)) &= (t(y)s(z)
    \otimes 1)\widetilde{\rT}(a \otimes b)(t(y') \otimes s(x')),
  \end{aligned}
\end{align}
and make the following analogues of \eqref{dg:tltr-compatible}, \eqref{dg:tltr-module}, \eqref{dg:tltr-multiplicative} and \eqref{dg:tltr-coassociative} commute:
\begin{gather} \label{dg:ltrt-compatible}
\xymatrix@R=0pt@C=40pt{
  & \ArA \otimes A \ar[rd]^(0.55){\id \otimes m^{\op}}  & \\
A \otimes A \otimes A \ar[ru]^{ \widetilde{\lT} \otimes \id} \ar[rd]_{\id \otimes
  \widetilde{\rT}} & & \ArA \\
& A \otimes \ArA,  \ar[ru]_(0.55){m\otimes \id} } \\
\label{dg:ltrt-module}
\xymatrix@C=15pt@R=15pt{A \otimes A \otimes A  \ar[r]^(0.45){\id \otimes \widetilde{\lT}}
  \ar[d]_{m \oo \id} & A \otimes \ArA \ar[d]^{m \otimes
    \id} & A \otimes A  \otimes A \ar[r]^(0.45){\widetilde{\rT} \otimes \id} \ar[d]_{\id
    \otimes m^{\op}}
  & \ArA \otimes A \ar[d]^{\id \otimes m^{\op}} \\
  A\otimes A \ar[r]^{\widetilde{\lT}} & \ArA, & A\otimes A
  \ar[r]^{\widetilde{\rT}} & \ArA, \\  
A \otimes A \otimes A
      \ar[r]^(0.6){\id \otimes m^{\op}} \ar[d]_{(\widetilde{\lT})_{13}}
      & A \otimes A \ar[dd]^{\widetilde{\lT}} & A \otimes A
      \otimes A \ar[r]^(0.6){m \otimes \id} \ar[d]_{(\widetilde{\rT})_{13}} & A \otimes A \ar[dd]^{\widetilde{\rT}} \\
\bAs \oo A \oo \bAt \ar[d]_{(\widetilde{\lT})_{12}} & & \bAs\otimes A \otimes \bAt
      \ar[d]_{(
        {\widetilde{\rT}})_{23}} \\  
\bAs \oo {^{B}A^{B}} \oo \bAt \ar[r]_(0.6){ \id \otimes m^{\op}_{B}} & \ArA &
\bAs \oo {_{B}A_{B}} \oo \bAt \ar[r]_(0.55){m_{B} \otimes \id}
      & \ArA }  \\ \label{dg:ltrt-coassociative} \xymatrix@R=18pt{ A \otimes A\otimes A \ar[r]^(0.45){\id
     \otimes \widetilde{\rT}} \ar[d]_{\widetilde{\lT} \otimes \id} &
A \otimes \ArA \ar[d]^{\widetilde{\lT} \otimes \id} \\ \ArA \otimes A \ar[r]^(0.45){\id
  \otimes\widetilde{\rT}} & \ArArA.  }
\end{gather}
\begin{notation} \label{notation:right} Similarly as in Notation \ref{notation:left}, we choose an algebra $C$ with an anti-isomorphism $\kappa \colon C \to B$, use this to regard $A$ as a $C$-module in various ways, and write
  \begin{align*}
    \Ac \oo \cA \quad \text{and} \quad \ctA \oo \cAt
  \end{align*}
  for the quotients of $A\oo A$ by the subspaces spanned by all elements of the form $t(\kappa(x))a \oo b - a \oo bt(\lambda(x))$ in the first case and $as(\kappa(x)) \oo b - a \oo s(\kappa(x))b$ in the second case. Again, the choice $C$ and $\kappa$ is irrelevant here and will only be fixed from Section \ref{section:hopf} on.  Note that the flip map $\Sigma_{(A,A)}$ \new descends \old to isomorphisms
\begin{align} \label{eq:flips}
  \Sigma_{(\Ab,\bA)} \colon \Ab\oo \bA \to A^{C} \oo {^{C}A} \quad \text{and} \quad
  \Sigma_{(\btA,\bAt)} \colon \btA \oo \bAt \to \Ac \oo \cA.
\end{align}
\end{notation}

In the notation above, the maps $\widetilde{\lT}$ and $\widetilde{\rT}$ \new descend \old to maps
\begin{align*}
\lT &\colon \cAs \otimes \csA \to \ArA &&\text{and} &
\rT &\colon \ctA \otimes \cAt \to \ArA,
\end{align*}
respectively, and  these maps satisfy pentagonal relations similar to those
given in Proposition \ref{prop:tltr-pentagon} if the corresponding assumptions hold.

Given a right multiplier bialgebroid $\mathcal{A}=(A,B,s,t,\Delta)$, one can, similarly as in the case of left multiplier bialgebroids  (see Proposition
\ref{prop:left-co-opposite}), reverse the
comultiplication and obtain a \emph{co-opposite} right multiplier bialgebroid
\begin{align*}
  \mathcal{A}^{\co} = (B^{\op},t,s,\Delta^{\co}).
\end{align*}
One can also reverse the multiplication of the underlying algebra $A$ to pass between left and right
multiplier bialgebroids as follows.
\begin{proposition} \label{prop:left-opposite} 
Let $\mathcal{A}=(A,C,s,t,\Delta)$ be a left
  multiplier bialgebroid.  Regard $s$ as an anti-homomorphism and $t$ as a homomorphism,
  respectively, from $C$ to $M(A^{\op})\subseteq R(A^{\op})$,  and write
  $(A^{\op})_{C}$ and $^{C}(A^{\op})$ for $A^{\op}$, regarded as a $C$-module
  via $x  \cdot a^{\op} =(s(x)a)^{\op}$ and $a^{\op} \cdot x =(t(x)a)^{\op}$, respectively. 
Then the map $a\otimes b\mapsto a^{\op} \otimes b^{\op}$ descends to a linear isomorphism
  \begin{align*}
\AlA \to (A^{\op})_{C} \otimes {^{C}(A^{\op})}, \quad w\mapsto w^{\op \otimes \op},    
  \end{align*}
there exists a  well-defined homomorphism
\begin{align*}
  \Delta^{\op} \colon A^{\op} \to (A^{\op})_{C} \overline{\times} {^{C}(A^{\op})}, \quad
  ((b^{\op}\otimes c^{\op})\Delta^{\op}(a^{\op})) &:=(\Delta(a)(b\otimes c))^{\op\otimes \op},
\end{align*}
and $\mathcal{A}^{\op}:=(A^{\op},C,t,s,\Delta^{\op})$
 is a right multiplier bialgebroid.  Its
  associated maps $\widetilde{\lT}^{\op}$ and $\widetilde{\rT}^{\op}$ are given by
 \begin{align*}
   \widetilde{\lT}^{\op}(a^{\op} \otimes b^{\op}) &= (\widetilde{\Tl}(a\otimes b))^{\op\otimes \op}, &
   \widetilde{\rT}^{\op}(a^{\op} \otimes b^{\op}) &= (\widetilde{\Tl}(a\otimes b))^{\op \otimes \op}.
 \end{align*}
 Conversely, for every right multiplier bialgebroid $\mathcal{A}$, there exists a unique left
 multiplier bialgebroid $\mathcal{A}^{\op}$ such that $\mathcal{A}=(\mathcal{A}^{\op})^{\op}$.
\end{proposition}
The notion of a counit carries over as follows.

Given a right multiplier bialgebroid $(A,B,s,t,\Delta)$ and morphisms $\phi \in \Hom(\bAt,\bB)$ and
$\psi \in \Hom(\bAs,\Bb)$, we can form slice maps
\begin{align*} \phi \new \odot \old \id &\colon \ArA \to A, \ a \otimes b \mapsto bs(\phi(a)), & \id \new \odot \old \psi &\colon \ArA \to A, \ a
  \otimes b \mapsto at(\psi(b)).
\end{align*}
\begin{definition}
  A \emph{right counit} for a right multiplier
bialgebroid $\mathcal{A}=(A,B,s,t,\Delta)$ is a map $\varepsilon
\colon A\to B$ that satisfies 
\begin{align}
  \label{eq:right-counit-bimodule}
  \varepsilon(at(y)) &= ya &&\text{and} &
  \varepsilon(as(x)) &= ax &\text{for all } a\in A,x,y\in B,
\end{align}
that is, $\varepsilon \in \Hom(\bAs,\Bb) \cap
\Hom(\bAt,\bB)$\sindex{epsilon}{$\varepsilon$}{left/right counit}, and
    \begin{align} \label{eq:right-counit} 
(\varepsilon \new \odot \old \id)(_{\rho} T(a\otimes b)) &=ba &&\text{and} &   (\id \new \odot
\varepsilon\old)(_{\lambda} T(a\otimes b)) &=ab && \text{for all } a,b \in A.
\end{align} 
\end{definition}
One easily verifies that right counits for a right multiplier bialgebroid
$\mathcal{A}=(A,B,s,t,\Delta)$ coincide with right counits for the co-opposite $\mathcal{A}^{\co}$
up the canonical linear identification of $B$ with $B^{\op}$, and with left counits for the opposite
$\mathcal{A}^{\op}$ up the canonical linear identification of $A$ with $A^{\op}$.

Proposition \ref{prop:left-counit-props}, Lemma \ref{lemma:left-counit-image} and Proposition
\ref{prop:left-counit-existence} have the following right-handed counterparts. If
$\mathcal{A}=(A,B,s,t,\Delta)$ is a right multiplier bialgebroid, then
\begin{enumerate}
\item a right counit for $\mathcal{A}$ is unique if $\mathcal{A}$ is left- or right-full in a sense
  similar as it was defined for left multiplier bialgebroids in Definition \ref{df:full};
\item if the map $\rT$ (or $\Tl$) is surjective, then
  \begin{align}
    \label{eq:right-counit-multiplicative}
    \varepsilon(ab) &= \varepsilon(s(\varepsilon(a))b) \quad (\text{or }
    \varepsilon(ab)=\varepsilon(t(\varepsilon(a))b), \text{ respectively});
  \end{align}
\item without much loss of generality, one can assume $\varepsilon$ to be surjective;
\item if the maps $(\lT,\rT)$ are bijective and $A=As(^{t}I)=At(^{s}I)$, where $^{s}I,^{t}I
  \subseteq B$ denote the linear span of the images of all $\phi \in \Hom(\bAs,\Bb)$ and $\psi\in
  \Hom(\bAt,\bB)$, respectively, then $\mathcal{A}$ has a unique right counit.
\end{enumerate}
In the unital case, right multiplier bialgebroids with right counits satisfying the equations in
\eqref{eq:right-counit-multiplicative} correspond to right bialgebroids \cite{boehm:hopf},
\cite{boehm:bijective}, that is, an analogue of Proposition \ref{prop:left-unital} holds.

\new Given a right multiplier bialgebroid $(A,B,s,t,\Delta)$, we use the generalized Sweedler notation as well, but 
put the subscripts in brackets, so that
\begin{align*}
  \Delta(a) &= a_{[1]} \otimes a_{[2]}, &
  \lT(a\otimes b) &= ab_{[1]} \otimes b_{[2]}, & \rT(a\otimes b) &= a_{[1]} \otimes ba_{[2]}
\end{align*}
and so on. \old
\section{Multiplier Hopf algebroids}

\label{section:hopf}

We now come to the main part of this article, where the left- and the right-handed concepts
introduced above get assembled into a two-sided \new structure.  

In detail, \old a multiplier bialgebroid will be given by a left multiplier
bialgebroid and a right multiplier bialgebroid
\begin{align*}
  \mathcal{A}_{C}=(A,C,s_{C},t_{C},\Deltalt) \quad \text{and} \quad
  \mathcal{A}_{B}=(A,B,s_{B},t_{B},\Deltart),
 \end{align*}
respectively,
 subject to the following assumptions.

First, \new $\mathcal{A}_{C}$ and $\mathcal{A}_{B}$ \old have the same underlying total algebra $A$. By assumption, this algebra is
non-degenerate on the left and on the right, so that we can form the two-sided multiplier algebra
$M(A)$ which is the target of the maps $s_{B},t_{B},s_{C},t_{C}$.

The second assumption will be used to make sense of the third one,
and    reads
\begin{align} 
  s_{B}(B)  &= t_{C}(C), & t_{B}(B)&=s_{C}(C).
\end{align}
Then the maps $S_{B} := t_{C}^{-1} \circ s_{B} \colon B\to C$ and $S_{C} := t_{B}^{-1} \circ s_{C}
\colon C\to B$ are anti-isomorphisms, but not necessarily inverse to each other.  To simplify
notation, we shall identify $B$ with the image $s_{B}(B)$ and $C$ with the image $s_{C}(C)$, that
is, we assume $s_{B}=\id_{B}$, $s_{C}=\id_{C}$ and write $S_{B}$ and $S_{C}$ for $t^{-1}_{C}$ and
$t^{-1}_{B}$, respectively. Furthermore, we denote elements of $B$ by $x,x',x'',\ldots$ and elements of
$C$ by $y,y',y'',\ldots$, and write $\bA, \ \Ab$ if we regard $A$ as a left or right module over $B$
via left or right multiplication, and $\sbA, \ \Asb$ if we regard $A$ as a right or left module over
$B$ via $a\cdot x = t_{B}(x)a$ or $x\cdot a=at_{B}(x)$, respectively. Likewise, we use the notation
$\cA,\ \Ac,\ \scA,\ \Asc$, respectively.

To formulate the third assumption, observe that $C$-bilinearity of $\Deltalt$, see
\eqref{eq:left-delta-bimodule}, and $B$-bilinearity of $\Deltart$, see
\eqref{eq:right-delta-bimodule}, now take the form
\begin{align} \label{eq:hopf-delta-bimodule} 
  \begin{aligned}
    \Deltalt(xyax'y') &= (y \otimes x)\Deltalt(a)(y' \otimes x'),
    & \Deltart(xyax'y') &= (y \otimes x)\Deltart(a)(y' \otimes x')
  \end{aligned}
\end{align} 
for all $a\in A$, $x,x' \in B$, $y,y' \in C$. Similarly, one can
rewrite the relations \eqref{eq:tltr-bimodule} and
\eqref{eq:ltrt-bimodule} for the canonical maps $\widetilde{\Tl},\widetilde{\Tr}$ and
$\widetilde{\lT}, \widetilde {\rT}$. Now, the third assumption is the following  mixed co-associativity,
 \begin{align} \label{eq:compatible}
     \begin{aligned} ((\Deltalt \otimes \id)((1 \otimes c)\Deltart(b)))(a \otimes 1 \otimes 1) &= (1
\otimes 1 \otimes c)((\id \otimes \Deltart)(\Deltalt(b)(a \otimes 1))), \\ (a \otimes 1 \otimes 1)((\Deltart
\otimes \id)(\Deltalt(b)(1 \otimes c))) &= ((\id \otimes \Deltalt)((a \otimes 1)\Deltart(b)))(1 \oo
1 \otimes c)
     \end{aligned}
\end{align} for all $a,b,c \in A$, which amounts to
commutativity of the following
diagrams,
\begin{align} \label{dg:compatible} 
  \begin{aligned}
    \xymatrix@C=20pt@R=15pt{ A\otimes A\otimes A \ar[r]^(0.45){\id \otimes \widetilde{\Tr}}
      \ar[d]_{\widetilde{\lT} \otimes \id} & A\otimes \AlA \ar[d]^{\widetilde{\lT} \otimes
        \id} \\ \ArA \otimes A \ar[r]^(0.45){\id \otimes \widetilde{\Tr}} & \bAs \oo {^{B}A^{C}} \oo \csA,
}      \qquad
   \xymatrix@C=20pt@R=15pt{
      A\otimes A\otimes A \ar[r]^(0.45){\id \otimes \widetilde{\rT}} \ar[d]_{\widetilde{\Tl} \otimes \id}
      & \ar[d]^{\widetilde{\Tl} \otimes \id} A\otimes \Asc \otimes \Ac \\ \AlA
      \ar[r]^(0.45){\id \otimes \widetilde{\rT}} & \AlArA,  }
  \end{aligned}
\end{align} \new
and in Sweedler notation to the relations
\begin{align}
(a_{(1)})_{[1]} \otimes (a_{(1)})_{[2]} \otimes a_{(2)} &= a_{[1]} \otimes (a_{[2]})_{(1)} \otimes (a_{[2]})_{(2)}, \label{eq:sweedler-mixed-coassociative-1} \\
  a_{(1)} \otimes (a_{(2)})_{[1]} \otimes (a_{(2)})_{[2]} &= (a_{[1]})_{(1)} \otimes (a_{[1]})_{(2)} \otimes a_{[2]} \label{eq:sweedler-mixed-coassociative-2} 
\end{align}
for all $a\in A$. \old
\begin{definition} \label{definition:mult-hopf-algebroid} A \emph{multiplier bialgebroid}
 $\mathcal{A}=(A,B,C,t_{B},t_{C},\Deltart,\Deltalt)$ consists of
\begin{enumerate}
\item a non-degenerate, idempotent algebra $A$,
\item subalgebras $B,C \subseteq M(A)$ with anti-isomorphisms $t_{B} \colon
  B\to C$ and $t_{C} \colon C\to B$, 
\item maps $\Deltalt\colon A \to \AltkA$ and $\Deltart \colon A \to \ArtkA$
\end{enumerate}
  such that
\begin{enumerate} \setcounter{enumi}{3}
\item $\mathcal{A}_{B}=(A,B,\id_{B},t_{B},\Deltart)$ is a right multiplier bialgebroid,
\item  $
\mathcal{A}_{C}= (A,C,\id_{C},t_{C},\Deltalt)$
is a left multiplier bialgebroid, and
\item   the mixed co-associativity conditions \eqref{eq:compatible} hold.
\end{enumerate}
We call left counits of $\mathcal{A}_{C}$ and right counits of $\mathcal{A}_{B}$ just left and right
counits, respectively, of $\mathcal{A}$. Likewise, we call the canonical maps $\Tl,\Tr$ of
$\mathcal{A}_{C}$ and $\lT,\rT$ of $\mathcal{A}_{B}$ just the canonical maps of $\mathcal{A}$.

\new
We call such a multiplier  bialgebroid $\mathcal{A}$ \emph{unital} if the  algebras $A,B,C$,
the inclusions $B,C\hookrightarrow A$ and the maps $\Deltalt,\Deltart$ are unital, that is, if $\mathcal{A}_{B}$ and $\mathcal{A}_{C}$ are unital. \old
\end{definition}
\new Note that we do not assume existence of counits. \old

To establish our main result and the key properties multiplier bialgebroids, we need to perform a fair amount of calculations involving the associated canonical maps.  We present these calculations and the key relations satisfied by the canonical maps in the form of commutative diagrams, where one can verify that all of the maps involved are well-defined on the underling tensor products.
Additionally, we write out the key relations in the generalised Sweedler notation \new  introduced at the end of    \S\ref{section:left} and of \S\ref{section:right}, respectively.  \old
\begin{definition} \label{definition:antipode}
An \emph{antipode} for a   multiplier bialgebroid $\mathcal{A}=(A,B,C,t_{B},t_{C}$, $\Deltart,\Deltalt)$  is a linear map $S \colon A \to \new M(A)\old$  satisfying the following conditions:
  \begin{enumerate}
   \item $S$ is an anti-homomorphism \new  such that $S(A)A=A=AS(A)$\old;
  \item the extension of $S$ to $M(A)$ satisfies $S\circ t_{B} = \id_{B}$ and $S\circ t_{C} = \id_{C}$;
  \item there exist a left counit $\epslt$ and a right counit $\epsrt$ for $\mathcal{A}$ such that the following diagrams commute, where the unlabelled maps are given by multiplication:
    \begin{align} \label{dg:antipode}
      \xymatrix@C=0pt@R=15pt{A\otimes A \ar[d]_{\widetilde{T_{\rho}}}
        \ar[rr]^{\epsrt \otimes \id} && B \otimes  A \ar[r] & A, & \quad& A\otimes A
        \ar[d]_{\widetilde{_{\lambda} T}} \ar[rr]^{\id \otimes \epslt} && A \otimes C \ar[r]
        &A. \\
        \AlA \ar[rrr]_{S \otimes \id} & \qquad&& \ar[u]  M(A)_{C}
        \otimes \cA && \ArA \ar[rrr]_{\id \otimes S} & \qquad && \ar[u]
        \Ab \otimes {_{B} M(A)}}
    \end{align}
  \end{enumerate}
  We call such an antipode \emph{invertible} if it maps $A$ bijectively to $A \subseteq M(A)$.
\end{definition}
Condition (1)  is equivalent to saying that $S$ is a morphism of bimodules
\begin{align*}
  S \colon {^{B}A^{B}} \to {_{B}M(A)_{B}} \quad\text{and} \quad
  S \colon {^{C}A^{C}} \to {_{C}M(A)_{C}},
\end{align*}
and
in Sweedler notation, commutativity of the diagrams \eqref{dg:antipode} amounts to the equations
\begin{align} \label{eq:sweedler-antipode} \epsrt(a)b &= S(a_{(1)})a_{(2)}b, & a\epslt(b) &= ab_{[1]}S(b_{[2]})
  \end{align} 
for all $a,b \in A$.  

\new We now come to the first main result of this article. Our proof follows a similar strategy as  the proof of the implication (iv)$\Rightarrow$(i) of Proposition 4.2 in \cite{boehm:bijective}, but will be purely diagrammatic. 
\begin{proposition} \label{proposition:canonical-antipode}
  Let $\mathcal{A}$ be a multiplier bialgebroid. Suppose that it has counits and that its canonical maps $\lT$ and $\Tr$ are bijective. Then it has a unique antipode $S$. \old
\end{proposition}
\begin{proof}
 Consider the compositions
  \begin{align*} \Sr &\colon \AlA \xrightarrow{\Tr^{-1}}
    \Ab \otimes \bA
\xrightarrow{\epsrt \otimes \id} \Bb \otimes \bA \to A, \\ \lS
&\colon \AlA
\xrightarrow{\lT^{-1}} \Ac \otimes \cA \xrightarrow{\id \otimes \epslt} \Ac \otimes \cC \to A.
  \end{align*} 
  If $S$ is an antipode for $\mathcal{A}$, then \ref{definition:antipode} (3) implies that $S(a)b = \Sr(a \otimes b)$ and $aS(b) = \lS(a \otimes b)$ for all $a,b\in A$. Therefore, the antipode is unique. Let us prove existence. The maps $\Sr$ and $\lS$ satisfy $a\Sr(b \otimes c)= \lS(a \otimes b)c$ for all $a,b,c\in A$ because the following diagram commutes,
 \begin{align} \label{eq:antipode} \xymatrix@C=40pt@R=15pt{ 
\ArAlA
\ar[r]^{\id \otimes \Tr^{-1}} \ar[d]_{\lT^{-1} \otimes \id}& 
\Ab \otimes {^{B}A_{B}} \otimes \bA \ar[d]_{\lT^{-1} \otimes \id}
\ar[r]^{\id \otimes \epsrt \otimes \id} & \Ab \oo {_{B}B_{B}} \oo \bA \ar[d]
 \\ \Ac \otimes {_{C}A^{C}} \otimes \csA
\ar[r]^{\id \otimes \Tr^{-1}} \ar[d]_{\id \otimes \epslt \otimes
  \id} & \Ac \otimes \cA_{B} \otimes \bA \ar[r]^{\mult_{C} \oo
\id} \ar[d]^{\id \otimes \mult_{B}} & \Ab \otimes \bA
 \ar[d]^{\mult_{B}} \\ \Ac \oo {_{C}C_{C}} \oo \cA  \ar[r] & \Ac \otimes \cA
\ar[r]^{\mult_{C}} & A. }
  \end{align}   
 Consequently, there exists a linear map $S \colon A \to M(A)$ such
  that
  \begin{align*}
  S(b)c = \Sr(b\otimes c) \quad \text{and} \quad aS(b) = \lS(a \otimes b)  
  \end{align*}
 for
  all $a,b,c\in A$. 
By construction,
  \begin{align} \label{eq:antipode-bimodule-2}
    \begin{aligned}
      S(t_{C}(y)a)b &= \Sr(t_{C}(y)a \oo b) = \Sr(a \oo yb) = S(a)yb, \\
      aS(bt_{B}(x)) &= \Sl(a \otimes bt_{B}(x)) = \Sl(ax \otimes b) =
      axS(b)
    \end{aligned}
  \end{align}
 for all $x\in B$, $y\in C$ and $a,b\in A$.

To see that $S$ is an anti-homomorphism, observe that  the following diagram commutes,
\begin{align*} \xymatrix@C=25pt@R=15pt{ 
A^{C} \oo {^{C}A^{C}} \oo \cA \ar[rrr]^(0.6){m_{C}^{\op} \otimes \id}
\ar[d]_{(\Rr)_{23}} &&& \AlA \ar[d]^{\Rr}  \\ 
 A^{C} 
\otimes {A_{B}} \otimes  {_{B,C}A} \ar[d]_{\id \oo \epsrt  \oo
\id} \ar[r]^(0.52){(\Rr)_{13}} & A_{B}^{C}\otimes \cAt \otimes \bA \ar[rr]^(0.6){m_{C}^{\op} \otimes \id}
\ar[d]_{ \id \otimes S_{B}\circ\epsrt \oo \id} && \Ab \otimes \bA 
\ar[d]^{\epsrt \otimes \id} \\ 
 A^{C} \oo \Bb \oo {_{B,C}A} \ar[d]
& A^{C}_{B} \oo \cC \oo \bA  \ar[d] && \Bb \oo \bA \ar[d]\\
 \bA \otimes \sbA
\ar[r]^(0.6){\Rr} & \Ab \otimes \bA \ar[r]^{\epsrt\otimes \id} & \Bb \oo \bA \ar[r]& A.
 }
\end{align*} 
Indeed, the upper rectangle commutes by \eqref{dg:tltr-multiplicative}, and the lower right square commutes because $\epsrt(ab)=\epsrt(\epsrt(a)b)$ for all $a,b\in A$; see \eqref{eq:right-counit-multiplicative}. Thus, $\Sr(ba)c = \Sr(a)\Sr(b)c$ for all $a,b,c\in
A$.

We  claim that $AS(A)=A=S(A)A$. Indeed, diagram \eqref{eq:antipode} shows that
\begin{align*} 
(\id \otimes \Sr)(\ArAlA) &=
  \Ab \otimes \bA, & (\lS \otimes \id)(\ArAlA)
&= \Ac \otimes \cA
\end{align*} 
because the maps $\Tr^{-1},\lT^{-1}$ and $m$ are surjective.  We apply $\epsrt \otimes \id$ or $\id \otimes
\epslt$, respectively,  use Lemma \ref{lm:counit-isit} and \eqref{eq:antipode-bimodule-2}, and  conclude
\begin{align*}
A=  \epsrt(A)A \subseteq
\Sr(\AlA)&= S(A)A, & A= A\epslt(A)
\subseteq \lS(\ArA)&=AS(A).
\end{align*}

Finally, the diagrams in \eqref{dg:antipode} commute by construction.
\end{proof}

\new
Conversely, given a multiplier  bialgebroid $\mathcal{A}$ with an antipode $S$, we shall construct an inverse to the canonical map $\Tr$. A similar construction will give the inverse of $\lT$.
\begin{lemma} \label{lemma:trdag}
 There exists a  unique linear map
\begin{align*}
  \Tr^{\dag} \colon \AlA \to \Ab \otimes \bA
\end{align*}
  such that the following diagrams commute:
  \begin{align} \label{dg:trdag}
    \xymatrix@R=20pt@C=45pt{ \scA  \oo \Asc \oo A \ar[r]^(0.55){\id \otimes m(S\otimes\id)} \ar[d]_{\rT \otimes \id}
    & \AlA  \ar[d]^{\Tr^{\dag}} &     \Ac \oo {_{C}\scA} \oo \cA \ar[r]^{\id \otimes \Tr^{\dag}} \ar[d]_{\lT \otimes \id} &  \Ac \oo  \cA_{B} \oo \bA \ar[d]^{m\otimes \id}
    \\ \ArA \oo A   \ar[r]^(0.55){\id \otimes m(S\otimes\id)} & \AbA &
       \Ab \oo {^{B}A^{C}} \oo \cA \ar[r]^(0.58){\id \otimes m(S\otimes \id)} &  \AbA. 
}
  \end{align}
\end{lemma}
In the unital case, we could just let $\Tr^{\dag}=(\id \otimes S) \circ \rT$ so that $\Tr^{\dag}(a) = a_{[1]} \otimes S(a_{[2]})$. In the non-unital case, this relation is not well-defined but captured by  commutativity of the diagrams above. Indeed, in Sweedler notation, these diagrams amount  to the relations
\begin{align*}
  \Tr^{\dag}(a \otimes S(b)c) =  a_{[1]} \otimes S(ba_{[2]})c \quad \text{and} \quad
ab_{[1]} \otimes S(b_{[2]})c = (a \otimes 1)\Tr^{\dag}(b \otimes c).
\end{align*}
\begin{proof}[Proof of Lemma \ref{lemma:trdag}]
Let $\Sr=m(S\otimes \id)$ as before.
 To prove existence of a map $\Tr^{\dag}$ that makes the first diagram commute, we need to show that whenever we have an element $\omega =\sum_{i} b_{i} \otimes c_{i} \otimes d_{i} \in    \scA \otimes \Asc \otimes A$ such that
\begin{align*}
(\id \otimes \Sr)(\omega) =
    \sum_{i} b_{i} \otimes S(c_{i})d_{i}   
\end{align*}
is zero in $\AlA$,  then also
 \begin{align} \label{eq:trdag-aux-1}
(\id \otimes \Sr)(\rT \otimes \id)(\omega) =
 \sum_{i} b_{i[1]} \otimes S(c_{i}b_{i[2]})d_{i}
 \end{align}
is zero in $\AbA$.
To this end,
  consider the following diagram:
  \begin{align*}
    \xymatrix@C=10pt@R=15pt{
      \Ac \oo \cA^{C} \oo {^{C}A} \oo A
      \ar[rrrr]^{\id \otimes \id \otimes \Sr} \ar[rd]^(0.55){\id \otimes \rT \otimes \id}
      \ar[ddd]_{\lT \otimes \id \otimes \id} &
      \ar@{}[rrd]|{(4)}    &&& \ar@{-->}[ld]_(0.55){\id \otimes \Tr^{\dag}} \ar[ddd]^{\lT \otimes \id}
\Ac \oo \cA^{C} \oo \cA
      \\
\ar@{}[rd]|{(1)}      &  \ar@{}[rrd]|{(2)}    
\Ac \oo \cA_{B} \oo {^{B}A}
\ar[rr]^{\id \otimes \id \otimes \Sr} \ar[d]_{m \otimes \id \otimes \id} & \quad & 
\Ac \oo \cA_{B} \oo \bA
\ar[d]^{m \otimes \id}  \ar@{}[rd]|{(5)}    \\
&
\Ab \oo {^{B}A} \oo A
\ar[rr]_{\id\otimes \Sr} \ar@{}[rrd]|{(3)}  &   & \Ab \oo \bA &\\
\Ab \oo {^{B}A}^{C} \oo {^{C}A} \oo A
\ar[rrrr]_{\id \otimes \id \otimes \Sr} \ar[ru]^{\id \otimes m^{\op} \otimes \id} &&&& \ar[lu]_{\id \otimes \Sr}
\Ab \oo {^{B}A}^{C} \oo \cA
}
\end{align*}
Cell (1) commutes by \eqref{dg:ltrt-compatible}, (2) trivially, and (3) because $S$ is an anti-homomorphism. 
Now,  we make the  argument above precise. Suppose that $w  =\sum_{i} b_{i} \otimes c_{i} \otimes d_{i}\in  A^{C} \oo {^{C}A} \oo A$ and $(\id \otimes \Sr)(w)$ is zero.  Since the outer square commutes, we  can conclude that
\begin{align*}
  (\id \otimes \id \otimes \Sr)(\lT \otimes \id \otimes \id)(a \oo w)  = \sum_{i} ab_{i[1]} \otimes b_{i[2]} \otimes S(c_{i})d_{i}
\end{align*}
is zero
for all $a \in A$, and  since  cells (1)--(3)  commute, also
\begin{align*}
  (m  \otimes \id)(\id \otimes \id \otimes \Sr)(\id \otimes \rT \otimes \id)(a \oo w) = \sum_{i} ab_{i[1]} \otimes S(b_{i[2]}) S(c_{i})d_{i}
\end{align*}
is zero
for all $a \in A$. But then also the element \eqref{eq:trdag-aux-1} is zero because $S$ is anti-multiplicative and $\Ab \oo \bA$ is non-degenerate as a left module over $A\otimes 1$. 
Hence, we can deduce that there exists a unique map $\Tr^{\dag}$ that makes the first diagram in \eqref{dg:trdag} and cell (4) in the diagram above commute.  As $\Sr$ is surjective, we can deduce that cell (5) and hence also the second diagram in \eqref{dg:trdag} commute.
\end{proof}

\begin{lemma}
  The following diagram commutes:
  \begin{align*}
    \xymatrix@R=18pt{
  \sbA  \oo {^{B}A^{C}} \oo \cA  \ar[r]^{\id \otimes \Tr^{\dag}} \ar[d]_{\Tl \otimes \id} &   \sbA \oo \Asb_{B} \oo \bA \ar[d]^{\Tl\oo\id}  \\
   \AlAlA  \ar[r]^{\id \otimes \Tr^{\dag}}&
\AlAbA.
    }
  \end{align*}
\end{lemma}
\begin{proof}
  This follows easily from commutativity of the first diagram in Lemma \ref{lemma:trdag},  commutativity of \eqref{dg:tltr-coassociative},  and surjectity of the map $\Sr$.
\end{proof}
\begin{lemma} \label{lemma:trdag-counit}
We have $  (m \circ \Tr^{\dag}) (a\otimes b)= \epslt(a)b$ for all $a,b\in A$.
\end{lemma}
\begin{proof}
  In Sweedler notation, the idea is that 
  $a\epslt(b)c =ab_{[1]}  S(b_{[2]})c = a((m \circ \Tr^{\dag})(b\otimes c))$ for all $a,b,c\in A$ by  \eqref{dg:trdag} and \eqref{dg:antipode}. More formally,
  consider the following diagram:
  \begin{align*}
    \xymatrix@C=30pt@R=20pt{ 
 \Ac \oo  \cA_{B} \oo \bA\ar[r]^{ m \otimes \id} &  \Ac \oo {_{C}\scA} \oo \cA \ar `r/0pt[r] [rd]^{m} & \\
     \Ac \oo {_{C}\scA} \oo \cA \ar[u]^{\id \otimes \Tr^{\dag}} \ar[r]^{\lT \otimes\id} \ar[d]_{\id \otimes \epslt \otimes \id} &     \Ab \oo {^{B}A^{C}} \oo \cA \ar[u]_{\id \otimes m(S\otimes \id)} \ar[d]^{m(\id \otimes S) \otimes \id}    & A \\
 \Ac \oo \cC_{C} \oo \cA \ar[r] & \AcA  \ar `r/0pt[r] [ru]_{m} &
    }
  \end{align*}
The upper and the lower rectangle commute because of \eqref{dg:trdag} and \eqref{dg:antipode}, respectively, and the rectangle on the right hand side commutes as well. Hence so does the outer cell.
\end{proof}
\begin{proposition} \label{proposition:antipode-canonical}
  Let $\mathcal{A}$ be a multiplier bialgebroid with an antipode $S$. Then its canonical maps $\lT$ and $\Tr$ are bijective. 
\end{proposition}
\begin{proof}
  We show that the map $\Tr^{\dag}$ defined above is inverse to $\Tr$. 

In the unital case, we can use 
\eqref{eq:sweedler-mixed-coassociative-1}, \eqref{eq:sweedler-antipode}, and find
  \begin{align*}
    (\Tr^{\dag}\circ \Tr)(a \otimes b) = (a_{(1)})_{[1]} \otimes S((a_{(1)})_{[2]})a_{(2)}b &= a_{[1]}\otimes S((a_{[2]})_{(1)})(a_{[2]})_{(2)}b \\ &= a_{[1]} \otimes \epsrt(a_{[2]})b = a_{[1]}\epsrt(a_{[2]}) \otimes b = a \otimes b
  \end{align*}
for all $a,b\in A$. Thus,  $\Tr^{\dag}\circ \Tr=\id$, and a similar calculation shows that $\Tr \circ \Tr^{\dag} =\id$.

In the general case, we proceed as follows. We first claim that $\Tr^{\dag} \circ \Tr=\id$.
  Diagrams \eqref{dg:compatible} and \eqref{dg:trdag} imply that the following diagram commutes:
\begin{align*}
  \xymatrix@C=35pt@R=20pt{
 \AcAbA \ar[r]^{\id \otimes \Tr} \ar[d]_{\lT\otimes \id} &    \Ac \oo {_{C}\scA} \oo \cA \ar[r]^{\id \otimes \Tr^{\dag}} \ar[d]_{\lT \otimes \id} & \Ac \oo  \cA_{B} \oo \bA  \ar[d]^{m \otimes \id}  \\
\ArAbA \ar[r]^{\id \otimes \Tr} &    \Ab \oo {^{B}A^{C}} \oo \cA\ar[r]^(0.55){\id \otimes m(S\otimes \id)} & \AbA
  }
\end{align*}
By \eqref{dg:antipode}, the lower composition maps $a\otimes b\otimes c$ to $a \otimes \epsrt(b)c = a\epsrt(b) \otimes c$, and by definition of the counit, precomposition with $\lT \otimes\id$ gives $m\otimes \id$. Therefore, $(m \otimes \id)(\id \otimes \Tr^{\dag})(\id \otimes \Tr) = m\otimes \id$.  Since $\AbA$ is non-degenerate as a left $A\otimes 1$-module,  the claim follows.

To see that $\Tr \circ \Tr^{\dag} = \id$, consider the following diagram:
\begin{align*}
  \xymatrix@C=30pt@R=20pt{
  \sbA  \oo {^{B}A^{C}} \oo \cA \ar[r]^{\id \otimes \Tr^{\dag}} \ar[d]_{\Tl \otimes \id} &  \sbA \oo \Asb_{B} \oo \bA \ar[r]^{\id \otimes \Tr} \ar[d]_{\Tl \otimes \id} & \sbA  \oo {^{B}A_{B}} \oo \bA \ar[d]^{m^{\op} \otimes \id} \\
  \AlAlA \ar[r]^{\id \otimes \Tr^{\dag}} & \ar[r]^{\id \otimes m} \AlAbA & \AlA
}
\end{align*}
The right cell commutes by \eqref{dg:tltr-compatible}.
To see that the left cell commutes, use commutativity of the first diagram in  \eqref{dg:trdag},  commutativity of \eqref{dg:tltr-coassociative},  and surjectity of the map $\Sr$. By Lemma \ref{lemma:trdag-counit}, the lower composition maps $a \otimes b \otimes c$ to $a \otimes \epslt(b)c = t_{C}(\epslt(b))a\otimes c$, and now a similar argument as above shows that $\Tr \circ \Tr^{\dag} = \id$.

A similar argument shows that the map $\lT$ is invertible. 
\end{proof}

Summarising, we find:
\begin{theorem} \label{theorem:hopf-algebroid}
  Let $\mathcal{A}$ be a multiplier bialgebroid. Then the following two conditions are equivalent:
  \begin{enumerate}
  \item $\mathcal{A}$ has an antipode.
  \item $\mathcal{A}$ has counits and its canonical maps $\Tr$ and $\lT$ are bijective.
  \end{enumerate}
If these conditions hold, then the antipode is unique and   the following diagrams commute:
\begin{align} \label{dg:galois-inverse}
    \xymatrix@R=20pt{ \scA \oo \Asc \ar[r]^{\id \otimes S} \ar[d]_{\rT}
    & \AlA \ar[d]^{\Tr^{-1}} &  &                            \sbA \oo \Asb \ar[r]^{ S \otimes \id} \ar[d]_{\Tl} &  \ar[d]^{\lT^{-1}} \ArA \\
   \ArA \ar[r]^{\id \otimes S} & \AbA,&& \AlA \ar[r]^{ S \otimes \id} & \AcA.
}  
    \end{align}
\end{theorem}
In Sweedler notation, commutativity of \eqref{dg:galois-inverse} amounts to the equations
  \begin{align*} 
\Rr(a \otimes S(b)) &= a_{[1]} \otimes
    S(ba_{[2]}), & \lR(S(a)
    \otimes b) &= S(b_{(1)}a) \otimes b_{(2)}
  \end{align*} 
  for all $a,b\in A$. 
\begin{proof}[Proof of Theorem \ref{theorem:hopf-algebroid}]
  Propositions \ref{proposition:canonical-antipode} and \ref{proposition:antipode-canonical} imply equivalence of (1) and (2). Suppose that both conditions hold. To see that the first diagram above commutes, use \eqref{dg:trdag} and the fact that $S$ is anti-multiplicative. Similar arguments imply that the second diagram commutes as well.
\end{proof}
 We adopt the following terminology:
\begin{definition}
A \emph{multiplier Hopf algebroid}   is a multiplier bialgebroid with an antipode.
\end{definition} 
In the next result, we use the flip maps defined in  \eqref{eq:flips} and their inverses, but omit subscripts to  simplify notation.
\begin{proposition} \label{proposition:hopf-aux-1} Let $\mathcal{A}$  be a multiplier Hopf algebroid with antipode $S$. Then the following diagrams commute:
    \begin{align*} \xymatrix@R=15pt@C=20pt{ \Ac \otimes \cA \ar[d]_{\lT}
\ar[r]^{\Sigma}     & \sbA \oo \Asb \ar[r]^{\Tl}&  \AlA & 
\Ab \otimes \bA \ar[d]_{\Tr} 
\ar[r]^{\Sigma}   &  \scA \oo \Asc \ar[r]^{\rT}& \ArA, \\ 
  \ArA \ar[rr]_{ \id \otimes S} && \Ab \otimes \bA \ar[u]_{\Tr}
  & \AlA \ar[rr]_{S\otimes \id} && \Ac \otimes \cA. \ar[u]_{\lT} }
\end{align*}
  \end{proposition}
  \begin{proof}
We only prove commutativity of the  first diagram. Using Sweedler notation, we can use  \eqref{eq:sweedler-mixed-coassociative-2},   \eqref{eq:sweedler-antipode} and \eqref{eq:sweedler-left-counit} to conclude that
\begin{align*}
  a_{(1)}(b_{[1]})_{(1)} \otimes a_{(2)}(b_{[1]})_{(2)}S(b_{[2]}) &= a_{(1)}b_{(1)} \otimes a_{(2)}(b_{(2)})_{[1]}S((b_{(2)})_{[2]})  \\ &= a_{(1)}b_{(1)} \otimes a_{(2)}\epslt(b_{(2)})  \\&= a_{(1)}t_{C}( \epslt(b_{(2)}))b_{(1)} \otimes a_{(2)} = a_{(1)}b\otimes a_{(2)}
\end{align*}
for all $a,b\in A$. More formally,
\old consider following diagram, 
\begin{align*}
 \xymatrix@R=15pt@C=25pt{
 &
\ArAlA\ar[r]^{\id \otimes S \otimes \id} \ar[rd]_(0.3){\id
      \otimes \Sr} & \Ab \otimes {_{B}A_{C}} \otimes \cA \ar[d]^{\id
      \otimes \mult_{C}} \ar[r]^{\Tr \otimes \id} & \ctA \oo {_{C}A_{C}} \oo \cA \ar[d]^{\id \otimes \mult_{C}} \\ \Ac
    \otimes {_{C}A^{C}} \otimes \cA \ar[rd]_(0.4){\Sigma
      \otimes \id \quad} \ar[r]^(0.55){\id \otimes \Rr} \ar[ru]^{\lT \otimes \id}
    & \Ac \otimes {_{C}A_{B}} \otimes \bA \ar[r]_{\mult_{C} \otimes \id} &
    \Ab \otimes \bA \ar[r]_{\Tr} & \AlA, \\ & 
    {A^{C,B}} \oo {^{B}A} \oo \cA \ar[r]_(0.5){\Tl \otimes \id} &  \ctA \oo {_{C}A_{C}} \oo \cA \ar[ru]_{\id \otimes \mult_{C}}
    &}
  \end{align*} 
  where $m_{C}$ denotes the multiplication map from $\Ac\otimes\cA$ to $A$.
 The lower cell commutes by \eqref{dg:tltr-multiplicative} and
  \eqref{dg:tltr-compatible},  the upper left cell by \eqref{dg:trdag}, and the upper right cell by \eqref{dg:tltr-module}.  Hence, the entire diagram commutes, showing that $(\Tr \otimes
 \id)(\id \otimes S)(\lT\otimes \id)=(\Tl \otimes \id)(\Sigma \otimes
 \id)$.
\end{proof}

\section{The regular case}    

\label{section:regular}

The antipode of a multiplier Hopf algebroid turns out to be invertible if and only if  a certain co-opposite multiplier bialgebroid is a multiplier Hopf algebroid as well or, equivalently, if all four canonical maps are bijective and some minor technical condition holds.  We prove the equivalence of these conditions and derive further relations between the  canonical maps and the antipode, most importantly, that the antipode reverses the comultiplications.

Given a  multiplier bialgebroid $(A,B,C,t_{B},t_{C},\Deltart,\Deltalt)$, define $_{C}I,I^{C} \subseteq C$ by
\begin{align*}
{_{C}I} &:= \lspan\{ \psi(a) : \psi
  \in \Hom(\csA,\cC), a \in A\}, \\
 {I^{C}} &:= \lspan\{ \phi(a) : \phi \in \Hom(\ctA,\Cc), a \in A\}
\end{align*}
as before, and  similarly  define $I_{B},{^{B}I} \subseteq B$. \old
\begin{definition} \label{definition:hopf-regular}
  We call a multiplier bialgebroid  $(A,B,C,t_{B},t_{C},\Deltart,\Deltalt)$ a \emph{regular multiplier Hopf   algebroid} if the following conditions hold: 
  \begin{enumerate}
  \item the subspaces $t_{C}({_{C}I}) A$, ${I^{C}} A$, $A t_{B}([\Ab])$ and  $A  [\Asb]$ are equal to $A$;
  \item the canonical maps $\Tl,\Tr,\lT,\rT$ are  bijective.
  \end{enumerate}
\end{definition}
\new This terminology is justified:
\begin{remark} \label{remark:regular-justified} Every regular multiplier Hopf algebroid has counits by Proposition \ref{prop:left-counit-existence} and hence is a multiplier Hopf algebroid by Theorem \ref{theorem:hopf-algebroid}.   Conversely, if $\mathcal{A}$ is a multiplier Hopf algebroid, then condition (1) holds by Lemma \ref{lm:counit-isit} and the right-handed counterpart, and the maps $\lT$ and $\Tr$ are bijective, but not necessarily $\Tl$ nor $\rT$.
\end{remark}

To establish the main result stated above, we will make use of the following  four-fold symmetry of multiplier bialgebroids. \old

Let $\mathcal{A}=(A,B,C,t_{B},t_{C},\Deltart,\Deltalt)$ be a multiplier bialgebroid. Write  $(A^{\co})^{B}$, $_{B}(A^{\co})$ and   $
(A^{\co})_{C}$, $^{C}(A^{\co})$ for $A$, regarded as a $B$-module or $C$-module such that
\begin{align*}
 a \cdot x&:= t_{C}^{-1}(x)a, &   x \cdot a &:= x a, & a\cdot y &:= ay, &
  y\cdot a &:= at_{B}^{-1}(y)
\end{align*}
for all $a\in A$, $x\in B$, $y\in C$. Then we can define flip maps
\begin{align*}
  \Sigma_{( \ctA,\cA)}& \colon \AlA \to (A^{\co})^{B} \oo {_{B}(A^{\co})}, & 
  \Sigma_{(\Ab,\bAt)} &\colon \ArA
  \to  (A^{\co})_{C} \oo {^{C}(A^{\co})}
\end{align*}
and  homomorphisms
\begin{align*}
   (\Deltalt)^{\co}&\colon A \to   (A^{\co})^{B} \overline{\times} {_{B}(A^{\co})}, &
   ((\Deltalt)^{\co}(a))(b\otimes c) &=   \Sigma_{( \ctA,\cA)}(\Delta_{C}(a)(c\otimes b)), \\   (\Deltart)^{\co} &\colon A \to  (A^{\co})_{C} \overline{\times} {^{C}(A^{\co})}, &
  (b\otimes c) ((\Deltart)^{\co}(a)) &=   \Sigma_{( \Ab,\bAt)}((c\otimes b)\Delta_{B}(a)).
\end{align*}
 By Proposition \ref{prop:left-co-opposite} and the right-handed analogue, the tuple
\begin{align*}
  \mathcal{A}^{\co}:=(A,C,B,t_{B}^{-1},t_{C}^{-1},(\Deltart)^{\co},(\Deltalt)^{\co})
\end{align*}
is a multiplier bialgebroid again.  We call it the \emph{co-opposite} of $\mathcal{A}$.  

Next, denote by $C^{\op},B^{\op}$ the images of $C$ and $B$ under the canonical identification $M(A)^{\op}\cong M(A^{\op})$, regard $t_{C}$ and $t_{B}$ as anti-isomorphisms between $C^{\op}$ and $B^{\op}$,  denote by $a \mapsto a^{\op}$ the canonical anti-isomorphism from $A$ to $A^{\op}$, and write  $(A^{\op})_{B^{\op}}$, $^{B^{\op}}(A^{\op})$ and  $(A^{\op})^{C^{\op}}$, $_{C^{\op}}(A^{\op})$ for $A^{\op}$, regarded as a $B^{\op}$-module or $C^{\op}$-module such that
\begin{align*}
a^{\op} \cdot x^{\op} &= (xa)^{\op}, & x^{\op}\cdot a^{\op} &= (t_{C}^{-1}(x)a)^{\op}, &
a^{\op} \cdot y^{\op} &= (at_{B}^{-1}(y))^{\op}, & y^{\op} \cdot a^{\op} &= (ay)^{\op}
\end{align*}
for all $x\in B$, $y\in C$, $a\in A$.
Then the map $a\otimes b\mapsto a^{\op} \otimes b^{\op}$ descends to isomorphisms
\begin{align*}
 \AlA &\to (A^{\op})_{B^{\op}} \oo {^{B^{\op}}(A^{\op})}, &
\ArA & \to (A^{\op})^{C^{\op}} \oo {_{C^{\op}}(A^{\op})},
\end{align*}
which we write as $w\mapsto w^{(\op \otimes \op)}$. Using these isomorphisms, we define homomorphisms
\begin{align*}
  (\Deltalt)^{\op} &\colon A^{\op} \to 
(A^{\op})_{B^{\op}} \overline{\times} {^{B^{\op}}(A^{\op})}, &
(b^{\op}\otimes c^{\op}) ((\Deltalt)^{\op}(a^{\op})) &=(\Deltalt(a)(b\otimes c))^{\op\otimes\op}, \\
(\Deltart)^{\op} &\colon A^{\op} \to 
(A^{\op})^{C^{\op}} \overline{\times} {_{C^{\op}}(A^{\op})}, &
((\Deltart)^{\op}(a^{\op})(b^{\op} \otimes c^{\op}) &=((b \otimes c)\Deltart(a))^{\op\otimes \op}.
\end{align*}
Using Proposition \ref{prop:left-opposite}, one verifies that
 \begin{align*}
 \mathcal{A}^{\op}=(A^{\op},B^{\op},C^{\op},t_{C}^{-1},t_{B}^{-1},(\Deltalt)^{\op},(\Deltart)^{\op})
 \end{align*}
 is a multiplier bialgebroid again. We call it the \emph{opposite} of $\mathcal{A}$.

Composing the two constructions, we obtain the \emph{bi-opposite} multiplier bialgebroid
\begin{align*}
  \mathcal{A}^{\op,\co}=(A^{\op},C^{\op},B^{\op},t_{C},t_{B},(\Deltalt)^{\op,\co},(\Deltart)^{\op,\co}).
\end{align*}

 In Sweedler notation, 
\begin{align*}
(\Deltalt)^{\co}(a) &= a_{(2)} \otimes a_{(1)}, &
(\Deltart)^{\op}(a^{\op}) &= a^{\op}_{[1]} \otimes a^{\op}_{[2]}, &
(\Deltart)^{\op,\co}(a^{\op}) &= a^{\op}_{[2]} \otimes a^{\op}_{[1]}, \\
(\Deltart)^{\co}(a)
&= a_{[2]} \otimes a_{[1]}, &
(\Deltalt)^{\op}(a^{\op}) &= a^{\op}_{(1)} \otimes a^{\op}_{(2)}, &
(\Deltalt)^{\op,\co}(a^{\op}) &= a^{\op}_{(2)}\otimes a^{\op}_{(1)}.
\end{align*}

If $ \epslt$ and $\epsrt$ are a left and a right counit of $\mathcal{A}$, then left and right counits of $\mathcal{A}^{\co}$, $\mathcal{A}^{\op}$ and $\mathcal{A}^{\op,\co}$ are given by
\begin{align} \label{eq:symmetry-counits}
  (\epslt)^{\co} &= S_{C} \circ \epslt, & (\epsrt)^{\op} &=  S_{B} \circ \epsrt,  &
  (\epsrt)^{\op,\co} &=\epsrt, \\
 (\epsrt)^{\co}&=S_{B} \circ
  \epsrt, & 
  (\epslt)^{\op} &= S_{C} \circ \epslt, &
  (\epslt)^{\op,\co} &=\epslt,
\end{align}
respectively, as one can easily check using Propositions \ref{prop:left-co-opposite} and \ref{prop:left-opposite}.

The proof of the following result is straightforward and  left to the reader. \new
\begin{lemma} \label{lemma:hopf-symmetry} 
Let $\mathcal{A}$ be a multiplier bialgebroid. 
  \begin{enumerate}
  \item If $S$ is an antipode for $\mathcal{A}$, then $S$ is an antipode for  $\mathcal{A}^{\op,\co}$. 
  \item If $\mathcal{A}$ is a regular multiplier Hopf algebroid, then so are $\mathcal{A}^{\co}$,   $\mathcal{A}^{\op}$ and $\mathcal{A}^{\op,\co}$.
  \item If $S$ is an invertible antipode for $\mathcal{A}$, then the inverse $S^{-1}$ is an invertible antipode for  $\mathcal{A}^{\co}$ and for $\mathcal{A}^{\op}$.
  \end{enumerate}
\end{lemma}
We can now state and  the second main result of this section.
\begin{theorem} \label{tm:hopf-characterization} 
Let $\mathcal{A}$ be  a multiplier bialgebroid. Then the following conditions are equivalent: 
\begin{enumerate}
\item $\mathcal{A}$  has an invertible antipode;
\item $\mathcal{A}$ is a regular multiplier  Hopf algebroid;
\item $\mathcal{A}$ is  a multiplier Hopf algebroid and $\mathcal{A}^{\op}$ is a multiplier Hopf algebroid; 
\item $\mathcal{A}$ is a multiplier Hopf algebroid and $\mathcal{A}^{\co}$ is a multiplier Hopf algebroid.
\end{enumerate}
\end{theorem} 
\begin{proof}
  (1)$\Rightarrow$(2):  By Theorem \ref{theorem:hopf-algebroid}, $\Tr$ and $\lT$ are invertible. Lemma \ref{lemma:hopf-symmetry},   and the same theorem, applied to $\mathcal{A}^{\op}$ or $\mathcal{A}^{\co}$, imply that $\rT$ and $\Tl$ are invertible as well.  Condition (1) in Definition \ref{definition:hopf-regular} holds by Lemma \ref{lm:counit-isit}.

  (2)$\Rightarrow$(3):  Use Lemma \ref{lemma:hopf-symmetry} and  apply Remark \ref{remark:regular-justified} to $\mathcal{A}$ and to $\mathcal{A}^{\op}$.

  (3)$\Rightarrow$(4): By Lemma \ref{lemma:hopf-symmetry}, $\mathcal{A}^{\co} = (\mathcal{A}^{\op})^{\op,\co}$  is  a multiplier Hopf algebroid.

  (4)$\Rightarrow$(1): Denote by $S$ and by $S^{\co}$ the antipodes of $\mathcal{ A}$ and of $\mathcal{A}^{\co}$, respectively, and write  $\Sr(a\otimes b)=S(a)b$ for all $a,b\in A$ as before. Then  \old
  the following diagram commutes, 
\begin{align*} \xymatrix@R=18pt{ \AlAsA \ar[r]^{\id \otimes \mult_{B}}
\ar[d]_{\Rl \otimes \id} & \AlA \ar[r]^{\Sr} & A \\ A^{B} \oo {^{B}A_{B}} \oo \bA \ar `d/0pt[d] [dr]^(0.3){\id \otimes \epsrt \otimes \id}  \ar[r]^{\id \otimes \Tr} & A^{B} \otimes {^{B}A^{C}} \oo \csA \ar[u]_{\mult_{B}^{\op} \otimes \id}
\ar[r]^(0.6){\id \otimes \Sr} & A^{B} \oo {_{B}A} \ar[u]_{\Sr} \\ & A^{B} \oo {_{B}B_{B}} \oo \bA \ar `r/0pt[r] [ru] & }
\end{align*} 
showing that $S(a)bc=S(S^{\co}(b)a)c=S(a)S(S^{\co}(b))c$ for all $a,b,c\in A$. We first conclude $S(S^{\co}(A)A) =S(A)A=A$, whence $S(A)=A$,  and $S(S^{\co}(b))=b$, and then by symmetry $S^{\co}(A)=A$ and $S^{\co}(S(b))=b$. Thus, $S\circ S^{\co}=\id_{A}$, and likewise $S^{\co} \circ S=\id_{A}$.
\end{proof}

\old
The canonical maps and the antipode satisfy the following useful relations.
\begin{corollary} \label{prop:antipode-aux} Let $\mathcal{A}=(A,B,C,t_{B},t_{C},\Deltart,\Deltalt)$ be a regular multiplier Hopf algebroid with antipode $S$ and canonical maps $\Tl,\Tr, \lT,\rT$. Then the following diagrams commute, where we omitted the subscripts on $\Sigma$ for better legibility:
    \begin{align*} \xymatrix@R=15pt@C=8pt{ \Ac \otimes \cA \ar[dd]_{\lT}
\ar[rd]_{\Tl\Sigma} \ar[rr]^{(S \otimes \id)\Sigma} && \ArA &
\Ab \otimes \bA \ar[dd]_{\Tr} \ar[rr]^{(\id \otimes S)\Sigma}
\ar[rd]_{\rT\Sigma} & & \AlA \\ 
 & \AlA & &  & \ArA & \\ \ArA \ar[rr]_{(S \otimes
\id)\Sigma} && \ctA \otimes \cAt \ar@{<-}[lu]^(0.6){\Sigma\Tr^{-1}}
\ar@{<-}[uu]_{\rT^{-1}}  & \AlA \ar[rr]_{(\id \otimes S)\Sigma} && \btA \otimes \bAt
\ar@{<-}[uu]_{\Tl^{-1}} \ar@{<-}[lu]^(0.6){\Sigma\lT^{-1}} }
\end{align*}
\end{corollary} 
\begin{proof} 
\new The lower left triangles commute by Proposition \ref{proposition:hopf-aux-1} in the
first square commutes. The same result applied to $\mathcal{A}^{\co}$ implies that the remaining triangles commute. \old
\end{proof}
\begin{remark}
  In Sweedler notation, commutativity of the first diagram in
  Proposition \ref{prop:antipode-aux} 
 amounts to equivalence of the following conditions  for arbitrary elements
$a_{i},b_{i},c_{j},d_{j} \in A$,
\begin{enumerate}
\item $\sum_{i} S(b_{i}) \otimes a_{i} = \sum_{j} c_{j[1]} \otimes d_{j}c_{j[2]}$,
\item $\sum_{i} a_{i(1)}b \otimes a_{i(2)} = \sum_{j} d_{j(1)} \otimes
d_{j(2)}c_{j}$,
\item $\sum_{i} a_{i}b_{i[1]} \otimes b_{i[2]} = \sum_{j} d_{j} \otimes
S^{-1}(c_{j})$,
\end{enumerate} and commutativity of the second diagram amounts to
equivalence of the following conditions,
\begin{enumerate}
\item[(1')] $\sum_{i} b_{i} \otimes S(a_{i}) = \sum_{j}  d_{j(1)}c_{j} \otimes
d_{j(2)}$,
\item[(2')] $\sum_{i} b_{[1]} \otimes ab_{[2]} = \sum_{j} d_{j}c_{j[1]} \otimes
c_{j[2]}$,
\item[(3')] $\sum_{i} a_{i(1)} \otimes a_{i(2)}b_{i} = \sum_{j} S^{-1}(d_{j}) \otimes
c_{j}$.
\end{enumerate}
\end{remark}
The antipode does not only reverse the multiplication, but also the comultiplication:
\begin{proposition} \label{prop:antipode-comult} Let $\mathcal{A}=(A,B,C,t_{B},t_{C},\Deltart,\Deltalt)$ be a regular multiplier Hopf algebroid with antipode $S$ and canonical maps $\Tl,\Tr,\lT,\rT$. Then the following diagrams commute, where we omitted the subscripts on $\Sigma$ for better legibility:
  \begin{align*} \xymatrix@C=30pt@R=15pt{ \btA \otimes \bAt \ar[d]_{\Tl}
\ar[r]^{\Sigma(S \otimes S)} & \ctA \otimes \cAt \ar[d]^{\rT} \\ \AlA \ar[r]_{\Sigma (S \otimes S)} & \ArA }
\quad \text{and} \quad \xymatrix@C=30pt@R=15pt{ \Ac \otimes \cA \ar[d]_{\lT} \ar[r]^{\Sigma(S \otimes S)}
& \Ab \otimes \bA \ar[d]^{\Tr} \\ \ArA \ar[r]_{\Sigma (S \otimes S)} & \AlA, }
  \end{align*}
\end{proposition}
\new  In Sweedler notation, commutativity of the diagrams above  amounts to the relations
  \begin{gather*} S(b)_{[1]} \otimes S(a)S(b)_{[2]} =    S(b_{(2)}) \otimes S(b_{(1)}a) = S(b_{(2)}) \otimes
    S(a)S(b_{(1)}), \\ S(b)_{(1)} \otimes S(b)_{(2)}S(a) =    S(b_{[2]}) \otimes S(ab_{[1]}) = S(b_{[2]}) \otimes
    S(b_{[1]})S(a)
  \end{gather*} for all $a,b\in A$. Note that in the non-unital case, the expressions on the
  right hand side require a suitable interpretation, which is given by
  the expressions in the middle. \old
\begin{proof}[Proof of Proposition
  \ref{prop:antipode-comult}] Combining the preceding result with the diagrams
  \eqref{dg:galois-inverse}, we find that the following diagram and
  hence the first square commute:
  \begin{align*} \xymatrix@R=15pt{ 
\btA \otimes \bAt \ar[d]_{\Tl} \ar[r]^{S \otimes \id}
& \ArA \ar[r]^{(S \otimes \id)\Sigma} & \ctA \otimes \cAt \ar[d]^{\rT} \\ \AlA \ar[r]_{S \otimes \id} & \Ac \otimes \cA
\ar[u]^{\lT} \ar[r]_{(S\otimes \id)\Sigma} & \ArA }
  \end{align*} 
  To obtain the second square, apply the same argument to $\mathcal{A}^{\co}$.
\end{proof}
The definition of an isomorphism between multiplier bialgebroids is
straightforward and left to the reader.
\begin{corollary} \label{corollary:hopf-symmetry}
  The antipode of a multiplier Hopf algebroid $\mathcal{A}$ is an
  isomorphism between $\mathcal{A}$ and the bi-opposite
  $\mathcal{A}^{\op,\co}$. In particular, the counits and antipode of $\mathcal{A}$ are related by $S_{C}\circ \epslt = \epsrt\circ S$ and  $S_{B}\circ \epsrt =
  \epslt \circ S$.
\end{corollary}
\begin{proof}
The first assertion follows easily from   Proposition \ref{prop:antipode-comult}, and implies that the composition 
$S_{B} \circ \epsrt \circ S^{-1}$ is a left counit and $S_{C} \circ \epslt \circ S^{-1}$    is  a right counit of
$\mathcal{A}^{\op,\co}$. But by \eqref{eq:symmetry-counits}, the  counits of $\mathcal{A}^{\op,\co}$ are just  $\epslt$ and $\epsrt$,
respectively.
\end{proof}

Let us finally comment on the relation to Hopf algebroids. 
\begin{proposition}
  Let $\mathcal{A}$ be a unital regular multiplier Hopf algebroid  with antipode $S$.  Then the left  and the right bialgebroid associated to $\mathcal{A}_{B}$ and $\mathcal{A}_{C}$,  respectively, form a Hopf algebroid.  Conversely, every  Hopf algebroid with invertible antipode arises this way from a unital regular multiplier Hopf  algebroid.
\end{proposition}
\begin{proof} Use Proposition 
 \ref{prop:left-unital} and its right-handed analogue, and note that the  conditions (1) and (2) in Definition \ref{definition:antipode} are equivalent to
conditions (iii) and (iv) of Definition 4.1 in \cite{boehm:hopf}.
\end{proof}

\section{Special cases and examples}

\label{section:examples}

\new
To keep this article moderately sized, we only discuss a few special cases and examples. Further examples related to dynamical quantum groups \cite{timmermann:free}, crossed products for braided-commutative Yetter-Drinfeld algebras \cite{militaru},  and Pontrjagin duality  can be found in \cite{timmermann:integration} and \cite{MR3607289}.
\old

\subsection{Multiplier Hopf algebroids associated with weak multiplier Hopf algebras}
Weak multiplier Hopf algebras were introduced by the second author and Wang in \cite{daele:weakmult0},  \cite{daele:weakmult2}, \cite{daele:weakmult} as non-unital versions of weak Hopf algebras.  The precise relation between \new regular \old weak multiplier Hopf algebras and \new regular \old multiplier Hopf algebroids is studied in \cite{daele:relation}. Briefly, one can associate to every regular weak multiplier Hopf algebra a regular multiplier Hopf algebroid as follows.

A \new  \emph{weak multiplier Hopf algebra}  as defined in  \cite[Definition 1.14]{daele:weakmult} \old consists of a non-degenerate, idempotent algebra $A$ and a
homomorphism $\Delta \colon A\to M(A\otimes A)$ satisfying the following conditions:
\begin{enumerate}
\item   for all $a,b\in A$, \new the elements $(a\otimes 1)\Delta(A)$ and $\Delta(a)(1 \otimes b)$  belong to $A\otimes A$; \old
  \item $\Delta$ is coassociative in the sense that for all $a,b,c\in A$,
      \begin{align*} (a\otimes 1 \otimes 1)(\Delta \otimes \iota)(\Delta(b)(1 \otimes c)) =
(\iota\otimes \Delta)((a \otimes 1)\Delta(b))(1 \otimes 1 \otimes c);
    \end{align*}
  \item $\Delta$ is full in the sense that there are no strict subspaces $V,W\subset A$
satisfying
      \begin{align*} \Delta(A)(1 \otimes A) &\subseteq V\otimes A &&\text{or} & (A \otimes
1)\Delta(A) &\subseteq A\otimes W;
    \end{align*}
\item there exists a linear map $\varepsilon\colon A\to \C$ called the \emph{counit} such
that for all $a,b\in A$,
     \begin{align*} (\varepsilon \otimes \iota)(\Delta(a)(1\otimes b)) &= ab= (\iota
\otimes \varepsilon)((a \otimes 1)\Delta(b));
    \end{align*}
  \item there exists an idempotent $E \in M(A \otimes A)$ such that
      \begin{align*} \Delta(A)(1 \otimes A) &=E(A \otimes A) &&\text{and} & (A \otimes
1)\Delta(A)&=(A\otimes A)E,
    \end{align*}
  \item the idempotent $E$ in condition (4) satisfies
      \begin{align*} (\Delta \otimes \iota)(E) =(E \otimes 1)(1\otimes E)=(1 \otimes
E)(E\otimes 1) = (\iota \otimes \Delta)(E),
    \end{align*}
   where $\Delta \otimes \iota$ and $\iota \otimes \Delta$ are extended
to homomorphisms $M(A \otimes A) \to M(A \otimes A \otimes A)$ such that $1 \mapsto E
\otimes 1$ or $1 \mapsto 1 \otimes E$, respectively;
\item  \new the kernels of the linear maps
     \begin{align*}
      \begin{aligned} T_{1} &\colon A\otimes A \to A\otimes A, & a\otimes b&\mapsto
\Delta(a)(1 \otimes b), \\ T_{2} &\colon A\otimes A \to A\otimes A, & a\otimes b&\mapsto
(a \otimes 1)\Delta(b), \\
      \end{aligned}
    \end{align*}
 are given by
\begin{align*}
  \ker T_{1}= (1-G_{1})(A \otimes A) \quad \text{and} \quad
\ker T_{2}= (1-G_{2})(A \otimes A),
\end{align*}
where $G_{1},G_{2} \colon A\otimes A \to A\otimes A$ are characterized by
  \begin{align*}
    (G_{1} \otimes \id)(\Delta_{13}(a)(1 \otimes b\otimes c)) &= \Delta_{13}(a)(1 \otimes E)(1 \otimes b\otimes c), \\
(\id \otimes G_{2})((a\otimes b\otimes 1)\Delta_{13}(c)) &= (a\otimes b\otimes 1)(E\otimes 1)\Delta_{13}(c)
  \end{align*}
for all $a,b,c\in A$, see \cite[Proposition 1.11]{daele:weakmult}. \old
 \end{enumerate}
Given a weak multiplier Hopf algebra $(A,\Delta)$ as above, there exists an
antipode, which is a \new linear map $S\colon A\to M(A)$ \old such that the maps
  \begin{align*}
    \begin{aligned} R_{1} &\colon A\otimes A \to A\otimes A, & a\otimes b &\mapsto a_{(1)} \otimes S(a_{(2)})b, \\ R_{2} &\colon A\otimes A\to A\otimes A, & a \otimes b
&\mapsto aS(b_{(1)}) \otimes b_{(2)}
    \end{aligned}
  \end{align*}
 are well-defined and satisfy $T_{i}R_{i}T_{i}=T_{i}$ and
$R_{i}T_{i}R_{i}=R_{i}$ for $i=1,2$; see \new \cite[Propositions 2.4, 2.7]{daele:weakmult}.   \old
Using this antipode, one defines source and target maps
$\varepsilon_{s},\varepsilon_{t}\colon A\to M(A)$ by
\begin{align*} \varepsilon_{s}(a) &= S(a_{(1)})a_{(2)}, & \varepsilon_{t}(a) &=a_{(1)}S(a_{(2)}).
\end{align*}
\new Let $(A,\Delta)$ be a weak multiplier Hopf algebra $(A,\Delta)$. Then $\Delta$ is \emph{regular} if   $\Delta(A)(A \otimes 1)$ and $(1\otimes A)\Delta(A)$ lie in $A \otimes A$ \cite[Definition 1.1]{daele:weakmult}, and $(A,\Delta)$ is \emph{regular} if the antipode $S$ is bijective \cite[Theorem 4.10]{daele:weakmult}.
\begin{theorem} \label{theorem:wmha-mhad} Let $(A,\Delta)$ be a  weak multiplier
Hopf algebra, where $\Delta$ is regular.
 Then there exists a multiplier Hopf algebroid $\mathcal{A}=(A,B,C,t_{B},t_{C},\Deltart,\Deltalt)$ \old such that
      \begin{align*} B&=\varepsilon_{s}(A), & C&=\varepsilon_{t}(A),
& t_{B} &= S^{-1}\circ
\iota_{B}, & t_{C} &= S^{-1}\circ
\iota_{C} &
\end{align*}
and, denoting by  $\pi_{C} \colon A\otimes A\to \AlA$ and $\pi_{B}
\colon A\otimes A\to \ArA$  the natural quotient maps,
  \begin{align*} \Deltalt(a)(1 \otimes b) & =\pi_{C}(\Delta(a)(1\otimes b)), & (a\otimes
1)\Deltart(b) &= \pi_{B}((a\otimes 1)\Delta(b))
  \end{align*}
for all $a,b\in A$. \new If $(A,\Delta)$ is regular, then so is $\mathcal{A}$. \old
\end{theorem} 
\begin{proof}
  \new For the regular case, the assertion is proven in \cite[\S 4]{daele:relation}. This proof carries over to the general case. \old
\end{proof}
In \cite[\S5]{daele:relation}, we also give necessary conditions for a regular multiplier Hopf algebroid
to arise from a  regular weak multiplier
Hopf algebra this way.

\subsection{Involutions}

\label{section:star}

\new Let us briefly  discuss involutions on multiplier bialgebroids and 
show that they behave with respect to counits and antipodes as one should expect from the
theory of (weak) multiplier Hopf algebras \cite{daele:0}, \cite{daele:weakmult}. \old

Suppose that $\mathcal{A}=(A,B,C,t_{B},t_{C},\Deltart,\Deltalt)$ is a multiplier bialgebroid and $A$ is  a $^{*}$-algebra, so that $M(A)$ is a $^{*}$-algebra with respect to the involution given by $T^{*}a=(aT^{*})^{*}$ and $aT^{*}=(Ta^{*})^{*}$. Assume that $B$ and $C$ are $^{*}$-subalgebras of $M(A)$ and that
\begin{align}
  \label{eq:involution}
  t_{B}\circ \ast \circ t_{C} \circ \ast &=\id_{C}, &  t_{C}\circ \ast \circ t_{B}
\circ \ast& =\id_{B}.
\end{align}
Then the formula $a \otimes b\mapsto a^{*} \otimes b^{*}$ defines mutually inverse conjugate-linear maps
\begin{align*} \Ab \otimes \bA &\rightleftarrows \ctA \otimes \cAt, &
  \Ac \otimes \cA &\rightleftarrows \btA \otimes \bAt, & \AlA &\rightleftarrows \ArA,
  \end{align*}
  and conjugation by $* \otimes *$ yields mutually inverse   conjugate-linear, multiplicative bijections $\End(\AlA)
  \rightleftarrows \End(\ArA)$, which restrict to mutually   inverse conjugate-linear, anti-multiplicative bijections
  \begin{align*}
    \AltkA \rightleftarrows \ArtkA.
  \end{align*}
We write these bijections as $T\mapsto T^{*}$. Then 
\begin{align} \label{eq:delta-involution}
  \Deltalt(a^{*}) &= \Deltart(a)^{*} \quad \text{for all } a\in A
\end{align}
if and only if the associated canonical maps satisfy
\begin{align}
  \label{eq:tltr-involution}
    (\ast \otimes \ast) \circ \Tl  &= \lT \circ (\ast \otimes \ast), &
    (\ast \otimes \ast) \circ \Tr &= \rT \circ (\ast \otimes \ast).  
\end{align}
\begin{definition} 
  We call a multiplier bialgebroid $\mathcal{A}=(A,B,C,t_{B},t_{C},\Deltart,\Deltalt)$ with an   involution on the underlying algebra $A$ a \emph{multiplier $^{*}$-bialgebroid} if $B$ and $C$ are   $^*$-subalgebras of $M(A)$ and the relations \eqref{eq:involution} and \eqref{eq:delta-involution}  hold. If $\mathcal{A}$ is also a multiplier Hopf algebroid, we call $\mathcal{A}$ a  multiplier Hopf $^*$-algebroid.
\end{definition}
A multiplier Hopf $*$-algebroid is automatically  regular. This follows from
   Theorem \ref{theorem:hopf-algebroid} and \eqref{eq:tltr-involution}, but also from the following relation for the antipode:
\begin{proposition} \label{prop:involution} Let
  $\mathcal{A}=(A,B,C,t_{B},t_{C},\Deltart,\Deltalt)$ be a multiplier Hopf
  $^{*}$-algebroid. Then its left and right counits $\epslt$ and $\epsrt$ and its antipode $S$ satisfy
  \begin{align*}
   \epsrt\circ * &= *\circ
  S_{C}\circ \epslt, & \epslt\circ *&=*\circ S_{B}\circ \epsrt, & 
  S\circ *\circ S \circ *&=\id_{A}.
  \end{align*}
\end{proposition}
\begin{proof} 
Denote by $\overline{V}$ the complex-conjugate of a vector space $V$ and by $\overline{f} \colon \overline{V}\to \overline{W}$ the  complex-conjugate of a linear map $f\colon V\to W$ of complex vector spaces. Then we obtain a regular multiplier Hopf algebroid $\overline{\mathcal{A}}=(\overline{A},\overline{B},\overline{C},\overline{t_{B}},\overline{t_{C}},\overline{\Delta_{B}},\overline{\Delta_{C}})$ with counits $\overline{\epslt},\overline{\epsrt}$ and antipode $\overline{S}$. The relations \eqref{eq:involution} and \eqref{eq:delta-involution} imply that the involution $*$ on $A$ defines an isomorphism from the complex-conjugate $\overline{\mathcal{A}}$ to the opposite $\mathcal{A}^{\op}$ of $\mathcal{A}$. 
Now, \eqref{eq:symmetry-counits} and Lemma \ref{lemma:hopf-symmetry} imply
$\overline{\epslt} = * \circ (\epslt)^{\co} \circ * = * \circ S_{B} \circ \epsrt \circ *$,
$\overline{\epsrt} = * \circ (\epsrt)^{\co} \circ * = * \circ S_{C} \circ \epslt \circ *$ and $\overline{S} = * \circ S^{-1} \circ *$, whence the claim follows.
\end{proof}

\subsection{The function algebra and the convolution algebra of an
  \'etale groupoid} Let $G$ be a locally compact, Hausdorff groupoid
that is \'etale in the sense that the  source and the target  map $s$
and $t$ from $G$ to the space of units $G^{0}$ are local homeomorphisms; see, for example, \cite{renault}.
Then the function algebra and the convolution algebra of $G$ can be
endowed with the structure of regular multiplier Hopf algebroids as follows.

Denote by $C_{c}(G)$ and $C_{c}(G^{0})$ the algebras of compactly supported continuous functions on
$G$ and on $G^{0}$, respectively, and denote by $s^{*},t^{*} \colon C_{c}(G^{0}) \to M(C_{c}(G))$
the pull-back of functions along $s$ and $t$, respectively, that is,
\begin{align*}
  (t^{*}(f)w)(\gamma) &=
f(t(\gamma))w(\gamma), & (s^{*}(f)w)(\gamma)
=f(s(\gamma))w(\gamma)
\end{align*}
for all $f\in C_{c}(G^{0})$, $w\in C_{c}(G)$ and $\gamma\in G$. Let
\begin{align*}
  A&=C_{c}(G), & B&=s^{*}(C_{c}(G^{0})),   & C &=t^{*}(C_{c}(G^{0})),
\end{align*}
and denote by $t_{B},t_{C}$ the isomorphisms $B \rightleftarrows C$ mapping $s^{*}(f)$ to $t^{*}(f)$
and vice versa.  Since $G$ is \'etale, the natural map $A \otimes A \to C_{c}(G\times G)$ factorizes
to an isomorphism
\begin{align*}
  \AlA &=  \ArA  \to C_{c}(\GstG),
\end{align*}
where $\GstG$ denotes the composable pairs of elements of $G$.
Denote by $\Deltalt, \Deltart  \colon C_{c}(G) \to M(C_{c}(\GstG))$ the pull-back
of functions along the groupoid multiplication, that is,
  \begin{align*} (\Deltalt(u)(v \otimes w))(\gamma,\gamma')
=u(\gamma\gamma')v(\gamma)w(\gamma') = ((v\otimes w)\Deltart(u))(\gamma,\gamma')
  \end{align*}
  for all $u,v,w\in A, \gamma,\gamma' \in G$.  The associated canonical maps $T_{\lambda}=\lT$ and
  $T_{\rho}=\rT$ are the transposes of the maps
\begin{align*}
  \GstG &\to \GttG,  \ (\gamma,\gamma') \mapsto (\gamma,\gamma\gamma'),
  &
  \GstG &\to \GssG, \ (\gamma,\gamma') \mapsto (\gamma\gamma',\gamma),
\end{align*}
respectively, and therefore bijective, where $\GpqG = \{ (\gamma,\gamma') \in G\times G :
p(\gamma)=q(\gamma')\}$.  The tuple $\mathcal{A}= (A,B,C,t_{B},t_{C},\Deltart,\Deltalt)$ is a
multiplier Hopf $^*$-algebroid with counits and antipode given by
\begin{align*}
  \epslt(w) &= t^{*}(w|_{G^{0}}), &
  \epsrt(w) &= s^{*}(w|_{G^{0}}), &
  (S(w))(\gamma) &=w(\gamma^{-1}) 
\end{align*}
for all $w\in C_{c}(G)$, as one can easily check.  Note that this multiplier Hopf $^*$-algebroid is
unital if and only if the groupoid $G$ is compact.

\smallskip

The space $C_{c}(G)$ can also be regarded as a $^*$-algebra with respect to the convolution product
and involution given by
\begin{align*}
  (u\ast v)(\gamma) &= \sum_{\gamma=\gamma'\gamma''}
  u(\gamma')v(\gamma''), &
  u^{*}(\gamma) &= \overline{u(\gamma^{-1})}.
\end{align*}
Since $G$ is \'etale, $G^{0}$ is closed and open in $G$, and the function algebra $C_{c}(G^{0})$
embeds into the convolution algebra $C_{c}(G)$. Denote by $\hat A$ this convolution algebra, let
$\hat B=\hat C=C_{c}(G^{0}) \subseteq \hat A$ and let $\hat t_{\hat B}=\hat t_{\hat
  C}=\id_{C_{c}(G^{0})}$. Then the natural map $A \otimes A \to C_{c}(G\times G)$ factorizes to
isomorphisms
\begin{align*}
  {\hat A^{\hat C}} \otimes {_{\hat C}\hat A} &\to C_{c}(\GttG), &
  {\hat A_{\hat B}} \otimes {^{\hat B}\hat A} &\to C_{c}(\GssG).
\end{align*}
Define $\hat \Delta_{\hat C}\colon C_{c}(G) \to \End(C_{c}(\GttG))$ and $ \hat \Delta_{\hat B}\colon
C_{c}(G) \to \End(C_{c}(\GssG))^{\op}$ by
\begin{align*}
  (\hat \Delta_{\hat C}(u)(v\otimes w))(\gamma',\gamma'') &=
  \sum_{t(\gamma)=t(\gamma')} u(\gamma)v(\gamma^{-1}\gamma')w(\gamma^{-1}\gamma''), \\
  ((v\otimes w)\hat \Delta_{\hat B}(u))(\gamma',\gamma'') &=
  \sum_{s(\gamma)=s(\gamma')} v(\gamma'\gamma^{-1})w(\gamma''\gamma^{-1})u(\gamma).
\end{align*}
Then $\widehat{\mathcal{A}}=(\hat A,\hat B,\hat C,\hat t_{\hat B},\hat
t_{\hat C},\hat \Delta_{\hat B},\hat \Delta_{\hat C})$ is a multiplier
Hopf $^*$-algebroid, and its counits and antipode are given by
\begin{align*}
  ({_{\hat C}\hat \varepsilon}(w))(\gamma) &= \sum_{r(\gamma')=\gamma } w(\gamma'), 
&
 (\hat \varepsilon_{\hat B}(w))(\gamma)&= \sum_{s(\gamma')=\gamma }w(\gamma'),
  &
  (\hat S(w))(\gamma'')&= w( \gamma''^{-1})
\end{align*}
for all $w\in C_{c}(G)$, $\gamma\in G^{0}$ and $\gamma''\in G$,
as one can easily check.

\smallskip

If $G$ is discrete, then $C_{c}(G)$ can also be regarded as a weak multiplier Hopf algebra with respect to the pointwise multiplication or convolution product, see \new  Examples 1.15 and 1.16  in \old \cite{daele:weakmult}, and then the multiplier Hopf algebroids $\mathcal{A}$ and $\widehat{\mathcal{A}}$ coincide with the ones obtained in Theorem \ref{theorem:wmha-mhad}.

\subsection{The tensor product $C\otimes B$} \label{subsection:tensor} Let $B$ and $C$ be non-degenerate and
idempotent algebras with anti-isomorphisms $S_{B} \colon B\to C$ and $S_{C}\colon C\to B$. Then the tensor
product $A:=C\otimes B$ is non-degenerate and idempotent again.  Identify $B$ and $C$ with their images in
$M(A)$ under the canonical inclusions and define $\Deltalt \colon A \to \End(\AlA)$ and
$\Deltart \colon A \to \End(\ArA)^{\op}$ by
\begin{align*}
  \Deltalt(y \otimes x)(a \otimes a') &= ya \otimes xa', &
 (a \otimes a')\Deltart(y \otimes x) &= ay
   \otimes a'x
    \end{align*} for all  $a,a'\in A$, $x \in B$, $y\in C$.
Then $\mathcal{A}=(A,B,C,S^{-1}_{C},S_{B}^{-1},\Deltart,\Deltalt)$ is a regular
multiplier Hopf algebroid with counits and antipode given by
\begin{align*}
\epslt(y\otimes x) &= yS_{B}(x), &
\epsrt(y\otimes x) &=S_{C}(y)x, &
S(y\otimes x) &=S_{B}(x)\otimes S_{C}(y)
\end{align*}
for all $x\in B$, $y\in C$.
The verification is straightforward, for example, the diagrams \eqref{dg:antipode} commute because 
for all $a\in A$, $x\in B$, $y\in C$,
\begin{align*}
  (m_{C} \circ (S\otimes \id)\circ T_{\rho})((y\otimes x)\otimes
  a) &= S_{C}(y) xa =  \epsrt(y\otimes x) a, \\
  (m_{B} \circ (\id \otimes S) \circ \lT)(a\otimes (y\otimes x)) &=
  ayS_{B}(x) = a\epslt (y\otimes x).
\end{align*}
If there exists a regular separability idempotent in $M(B\otimes C)$, then the algebra $A$ can be
equipped with the structure of a weak multiplier Hopf algebra, see \cite{daele:weakmult2}, and again
the multiplier Hopf algebroid $\mathcal{A}$ is isomorphic to the one obtained in Theorem
\ref{theorem:wmha-mhad}.

\subsection{A two-sided crossed product} The following construction generalizes Example
2.6 in \cite{vainer}, Example 3.4.6 in \cite{boehm:hopf} and  the preceding
example, and involves  actions of regular multiplier Hopf algebras, for which we refer to \cite{daele:actions}.

Let $B$ and $C$ be non-degenerate, idempotent algebras with anti-isomorphisms $S_{B} \colon B\to C$
and $S_{C} \colon C \to B$ again, and let $H$ be a regular multiplier Hopf algebra with a unital
left action on $C$ and a unital right action on $B$ such that the following conditions hold:
\begin{enumerate}
\item $B$ and $C$ are $H$-module algebras, that is, for all $h\in H$, $x,x'\in B$, $y,y'\in C$,
  \begin{align*}
  (xx') \actleft h = (x \actleft h_{(1)})(x'
  \actleft h_{(2)}) \quad \text{and}  \quad
  h\actright (yy') = (h_{(1)} \actright y)(h_{(2)} \actright y');
  \end{align*}
\item if $S_{H}$ denotes the antipode of $H$, then  for all $x\in B,y\in C, h\in H$,
  \begin{align*} S_{B}(x \actleft h) &= S_{H}(h) \actright S_{B}(x) &&\text{and} & S_{C}(h \actright
    y) &= S_{C}(y) \actleft S_{H}(h).
  \end{align*}
\end{enumerate}
    Then the space $A=C \otimes H \otimes B$ becomes a non-degenerate, idempotent algebra with
    respect to the product
      \begin{align} \label{eq:chb} (y \otimes h \otimes x)(y' \otimes h' \otimes x') &=
y(h_{(1)} \actright y') \otimes h_{(2)}h'_{(1)} \otimes (x \actleft h'_{(2)})x',
    \end{align}
    as can be seen using similar arguments as in \cite[\S5]{daele:actions}.  The algebras $C,H,B$
    embed naturally into $M(A)$. We identify them with their images in $M(A)$, and then the products
    \begin{align*}
      yhx &= y\otimes h \otimes x,&
      yxh &= y \otimes h_{(1)} \otimes (x\actleft h_{(2)}), &
      hyx &= (h_{(1)} \actright y) \otimes h_{(2)} \otimes x
    \end{align*}
    lie in $A\subseteq M(A)$ for all $x\in B$, $y\in C$ and $h\in H$.  Define $\Deltalt \colon A
    \to \End(\AlA)$ and $\Deltart \colon A \to \End(\ArA)^{\op}$ by
    \begin{align*}
 \Deltalt(yhx)(a \otimes a') &= yh_{(1)}a \otimes h_{(2)}xa', &
(a\otimes a')\Deltart(yhx) &= ayh_{(1)} \otimes a'h_{(2)}x
    \end{align*} 
    for all $x\in B$, $y\in C$, $h\in H$, $a,a'\in A$. Note that here, the legs of $h$ are covered by $a$ or $a'$, respectively. Then $\mathcal{A}=(A,B,C,S^{-1}_{C},S_{B}^{-1},\Deltart,\Deltalt)$ is a regular multiplier Hopf algebroid with counits and antipode given by
    \begin{align*}
       \epslt(yxh) &= yS_{B}(x)\epsh(h), & \epsrt(hyx)
&=S_{C}(y)x\epsh(h), &
S(yhx) &= S_{B}(x)S_{H}(h)S_{C}(y)
    \end{align*} 
    for all $x\in B$, $y\in C$, $h\in H$.

    The verification is a bit more tedious than in the previous examples, but straightforward
    again. For example, for all $x\in B$, $y\in C$, $a\in A$,
    \begin{align*}
      (\epslt \odot\id)(\widetilde{\Tr}(yxh \otimes a)) &=
      y \varepsilon_{H}(h_{(1)}) \otimes h_{(2)}x a = y\otimes xha \mapsto yxha, \\
      (\id \odot \epsrt)(\widetilde{\lT}(a\otimes hyx)) &=
      ah_{(1)}y \otimes \varepsilon_{H}(h_{(2)})x =ahy\otimes x \mapsto ahyx, \\
      (m \circ (S \otimes \id) \circ \Tr)(hxy \otimes a) &=
      S(h_{(1)}y)h_{(2)}xa = S_{C}(y)S_{H}(h_{(1)})h_{(2)}xa =
      \epsrt(hxy)a, \\
      (m\circ (\id \otimes S) \circ \lT)(a\otimes yxh) &=
      ayh_{(1)}S(xh_{(2)}) = ayh_{(1)}S_{H}(h_{(2)})S_{B}(x) = a     \epslt(yxh).
    \end{align*}

    If there exists a regular separability idempotent in $M(B\otimes C)$ that is compatible with the actions
    of $H$ on $B$ and $C$, then the algebra $A$ can also be equipped with the structure of a weak multiplier
    Hopf algebra, see \cite{daele:weakmult2}, and again the multiplier Hopf algebroid $\mathcal{A}$ is
    isomorphic to the one obtained in Theorem \ref{theorem:wmha-mhad}.

\new \subsection*{Co-commutative, proper and \'etale multiplier Hopf algebroids}
A special class of multiplier Hopf algebroids which includes the convolution algebras of  \'etale Hausdorff groupoids was introduced in  \cite{MR1836002} under the name  \emph{\'etale Hopf algebroids}.  We show that these are precisely the co-commutative and proper multiplier Hopf algebroids. Let us use the notation introduced in the beginning of section  \ref{section:regular}.
\begin{definition}
 A multiplier bialgebroid $\mathcal{A}=(A,B,C,t_{B},t_{C},\Delta_{B},\Delta_{C})$ is \emph{co-com\-mutative} if it is equal to its co-opposite
 $\mathcal{A}^{\co} = (A,C,B,t_{B}^{-1},t_{C}^{-1},(\Delta_{B})^{\co},(\Delta_{C})^{\co})$.
\end{definition}
\begin{remarks}  \label{remarks:co-commutative}
Let $\mathcal{A}$ be a co-commutative multiplier bialgebroid as above.
  \begin{enumerate}
  \item Evidently $B=C$, and this algebra is commutative.   
  \item The maps  $t_{B}$ and $t_{C}$ are involutive in the sense that $t_{B}=t_{B}^{-1}$ and $t_{C}=t_{C}^{-1}$. 
  If $\mathcal{A}$ has counits $\epslt$ and $\epsrt$, then  $t_{B}=\id_{B}$ and $t_{C}=\id_{C}$. For example,
    \begin{align*}
      z \epslt(a)b = \epslt(za)b = \epslt(a)t_{C}^{-1}(z)b    = t_{C}^{-1}(z) \epslt(a)b
    \end{align*}
    for all $a,b\in A$ and $z \in C=B$, whence $z=t_{C}^{-1}(z)$ for all $z\in C$  by Lemma \ref{lm:counit-isit}.
  \item If $\mathcal{A}$ is a multiplier Hopf algebroid, then it is regular by Theorem \ref{tm:hopf-characterization}, and its antipode $S$ is involutive in the sense that $S^{2}=\id$ by Lemma \ref{lemma:hopf-symmetry}.
  \end{enumerate}
\end{remarks}

Recall that a groupoid $G$ is \emph{proper} if  the map $G \to G^{0}\times G^{0}$ given by $\gamma \mapsto (t(\gamma),s(\gamma))$ is proper. For a multiplier bialgebroid, we define the corresponding property as follows:
\begin{definition}
   A multiplier bialgebroid $(A,B,C,t_{B},t_{C},\Delta_{B},\Delta_{C})$ is  \emph{proper} if  $BC\subseteq A$.
\end{definition}
 Given a multiplier bialgebroid, we define
the  Takeuchi products $\scA \times \cA \subseteq \scA \otimes \cA$ and $\Ab \times \Asb \subseteq \Ab \otimes \Asb$ as in the unital case, see \eqref{eq:left-takeuchi} and \eqref{eq:right-takeuchi}, respectively, and identify these with subalgebras of $\AltkA$ and $\ArtkA$ in the natural way.
\begin{lemma} \label{lemma:proper-cocommutative}
  Let $\mathcal{A}=(A,B,C,t_{B},t_{C},\Delta_{B},\Delta_{C})$ be a proper, co-commutative bialgebroid. Then
  \begin{align*}
    B &= C \subseteq A, & \Delta_{C}(A) &\subseteq \scA \times \cA, & \Delta_{B}(A) &\subseteq \Ab \times \Asb.
  \end{align*}
If $\mathcal{A}$ has counits, then they restrict to the identity on $B=C \subseteq A$.
\end{lemma}
\begin{proof}
  Clearly, $B=C=BC \subseteq A$. We show that $\Delta_{C}(A) \subseteq \scA \times \cA$. Let $a,b\in A$ and $y\in C$.  Then $\Delta_{C}(ay)(1 \otimes b) = \Delta_{C}(a)(y\otimes 1)(1 \otimes b)$, and since $\Delta_{C}(a)(y\otimes 1) \in \Delta_{C}(A)(A\otimes 1) = \scA \otimes \cA$, we can conclude that $\Delta_{C}(ay) \in \scA \times \cA$. But $AC=A$ and hence $\Delta_{C}(A) \subseteq \scA \times \cA$. A similar argument shows that $\Delta_{B}(A) \subseteq \Ab \times \Asb$. Finally, suppose that $\epslt$ and $\epsrt$ are counits for $\mathcal{A}$. Taking $a=y \in C$ in \eqref{eq:left-counit}, we find
  \begin{align*}
    yb = (\epslt \otimes\id)(\Tr(y\otimes b)) = \epslt(y \otimes b) = \epslt(y)b
  \end{align*}
for all $b\in A$ and hence  $(\epslt)|_{C}=\id$.  A similar argument shows that $(\epsrt)|_{B}=\id$.
\end{proof}
Recall that an \emph{\'etale Hopf algebroid}  \cite{MR1836002,MR1933720}  consists of 
\begin{itemize}
\item[(E1)] a total algebra $A$ with a commutative subalgebra $A_{0} \subseteq A$ in which $A$ has local units,
\item[(E2)] a co-commutative coalgebra structure $(\Delta,\varepsilon)$ on  $A$, regarded as an $A_{0}$-module with respect to right multiplication,
\item[(E3)] a linear involution $S\colon A \to A$ 
\end{itemize}
such that
\begin{itemize}
\item[(E4)] $\varepsilon|_{A_{0}} = \id$,  and $\varepsilon(a'a) = \varepsilon(\varepsilon(a')a)$ for all $a,a' \in A$;
\item[(E5)] $\Delta(y) = y \otimes 1 = 1\otimes y$ for all $y\in A_{0}$, and $\Delta(a'a)=\Delta(a')\Delta(a)$ for all $a,a' \in A$;
\item[(E6)] $S|_{A_{0}} = \id$, and $S(a'a) = S(a)S(a')$  for all $a,a' \in A$;
\item[(E7)] if $\Delta(a) = \sum_{i} a'_{i} \otimes a''_{i}$, then $\Delta(S(a)) = \sum_{i}S(a'_{i}) \otimes S(a''_{i})$;
\item[(E8)] $(\id \otimes S)\circ \lT \circ (\id \otimes S) \circ \lT = \id$, where $\lT \colon A_{A_{0}} \otimes {_{A_{0}}A} \to A_{A_{0}} \otimes A_{A_{0}}$ is given by $a \otimes b \mapsto (a\otimes 1)\Delta(b)$.
\end{itemize}
\begin{proposition} \label{proposition:mrcun}
  Let $\mathcal{A}=(A,B,C,t_{B},t_{C},\Delta_{B},\Delta_{C})$ be a proper, co-commutative multiplier Hopf algebroid, where $A$ has local units in $B$. Denote by $\epsrt$ and $S$ its right counit and its antipode, respectively. Then $(A,B,\Delta_{B},\epsrt,S)$ is an \'etale Hopf algebroid. Conversely, every \'etale Hopf algebroid arises this way.
\end{proposition}
\begin{proof}
  We first show that  $(A,B,\Delta_{B},\epsrt,S)$ is an \'etale Hopf algebroid.
  Lemma \ref{lemma:proper-cocommutative} implies $\Delta_{B}(A) \subseteq \Ab \times \Asb \subseteq \Ab \otimes \Asb= \Ab \otimes \Ab$. Clearly, $(\Delta_{B},\epsrt)$ forms a co-commutative coalgebra structure on $\Ab$ satisfying (E5). Assumption (E4)  holds by Lemma \ref{lemma:proper-cocommutative} and (\ref{eq:right-counit-multiplicative}),  (E3) and (E6)  by  Remarks \ref{remarks:co-commutative}, (E7) by Proposition \ref{prop:antipode-comult} and (E8) by Theorem \ref{theorem:hopf-algebroid}.

  Conversely, let $(A,A_{0},\Delta,\varepsilon,S)$ be an \'etale Hopf algebroid.  Then (E1), (E2) and (E5) imply that $(A,A_{0},\id_{A_{0}},\Delta)$ is a right multiplier bialgebroid, and (E3) and (E6) imply that with $\Delta':=(S \otimes S) \circ \Delta \circ S^{-1}$, the tuple $(A,A_{0},\id_{A_{0}},\Delta')$ is a left multiplier bialgebroid. Now, $\mathcal{ A}=(A,A_{0},A_{0},\id_{A_{0}},\id_{A_{0}},\Delta,\Delta')$ satisfies the mixed co-associativity conditions by (E7) and therefore is a multiplier bialgebroid. It is proper by (E1), co-commutative by (E2) and (E7), and its canonical map $\lT$ is invertible by (E8).  By co-commutativity and definition of $\Delta'$, the other three canonical maps of $\mathcal{A}$ are invertible as well.
Finally, (E1), (E4) and Theorem \ref{tm:hopf-characterization} imply that $\mathcal{A}$ is a multiplier Hopf algebroid.
\end{proof}

\section*{Acknowledgements}
We would like to thank the referee for careful reading and for very valuable comments, in particular, on non-regular multiplier Hopf algebroids and on \'etale Hopf algebroids.

\old

\def\cprime{$'$}

\end{document}